\documentclass[a4paper]{article}

\usepackage[numbers]{natbib}
\usepackage{lineno}
\usepackage{dsfont}
\usepackage{tikz}
\usetikzlibrary{arrows.meta, positioning}
\PassOptionsToPackage{hyphens}{url}\usepackage{hyperref}
\usepackage{float}
\usepackage{amsmath}
\usepackage{amssymb}
\usepackage{amsthm}

\usepackage{authblk}

\newtheorem{lemma}{Lemma}[section]
\newtheorem{theorem}[lemma]{Theorem}
\newtheorem{corollary}[lemma]{Corollary}

\newtheorem{remark}{Remark}

\DeclareMathOperator{\E}{\mathbb{E}}
\DeclareMathOperator{\prob}{\mathbb{P}}

\DeclareMathOperator{\Cov}{Cov}

\DeclareMathOperator{\1}{\mathds{1}}
\newcommand{\B}{\mathbb{B}}
\newcommand{\R}{\mathbb{R}}
\newcommand{\N}{\mathbb{N}}

\newcommand{\U}{\mathcal{U}}

\newcommand{\G}{\mathcal{G}}
\newcommand{\F}{\mathcal{F}}

\begin{document}

\title{A Non-Markovian Approach to a Stochastic Rumor Dynamics with Cognitive Deliberation}

\author[1]{Cristian F. Coletti\thanks{cristian.coletti@ufabc.edu.br}}
\author[2]{Denis A. Luiz\thanks{denisalu@ime.unicamp.br}}

\affil[1]{Centro de Matemática, Computação e Cognição, Universidade Federal do ABC, Av. dos Estados, 5001, Santo André, 09210-580, São Paulo, Brazil}
\affil[2]{Instituto de Matemática, Estatística e Computação Científica (IMECC), Unicamp, Rua Sérgio Buarque de Holanda, 651, Campinas, 13083-859, São Paulo, Brazil}

\date{} 

\maketitle

\begin{abstract}
We introduce a non-Markovian rumor model on a complete graph of $n$ vertices, integrating the classical interactional framework of Daley and Kendall (1964) with modern cognitive insights into misinformation. Unlike traditional Markovian models, our approach incorporates a deliberation delay -- a decision-making window where individuals evaluate information before committing to dissemination or refutation. We establish a Functional Law of Large Numbers (FLLN) and a Functional Central Limit Theorem (FCLT) to characterize the asymptotic behavior and diffusion-scaled fluctuations of the process.

\vspace{0.5cm}
\end{abstract}

\section{Introduction}
The phenomenon of rumor dissemination is a contemporary problem emerging in the context of applied social sciences, with implications for various topics such as politics, economics, and public health \cite{financialrumors,vaccinerumors,electoralrumor}. The first models used to better understand the dissemination of rumors were epidemiological models such as the SIR model, which considers the individuals in a population as susceptible, infected, or recovered \cite{kermack1927contribution}. In epidemiological analogies, while two infected individuals do not interact, two rumor spreaders affect the dissemination of information whenever they meet.

In 1964, Daley and Kendall \cite{daley1964epidemics} introduced what is commonly referred to as the first mathematical model of rumor spreading, the DK model. In this model, the population of $n+1$ individuals is divided into three disjoint classes: \textit{ignorants} (unaware of the rumor), \textit{spreaders} (actively propagating the rumor) and \textit{stiflers} (aware of the rumor but ceasing its dissemination). Denote the number of ignorants, spreaders and stiflers at time $t$ by $X^n(t),Y^n(t)$ and $Z^n(t)$, respectively, and assume that $X^n(t)+Y^n(t)+Z^n(t)=n+1$. The process $\{(X^n(t),Y^n(t))\}_{t\geq0}$ is a continuous-time Markov Chain (CTMC) with transitions and associated rates specified by
\begin{equation*}
    \begin{array}{cc}
         \text{transition}&\text{rate}  \\
         (-1,1)&X^nY^n\\
         (0,-2)&\binom{Y^n}{2}\\
         (0,-1)&Y^n(n+1-X^n-Y^n).
    \end{array}
\end{equation*}

Another classical model is the MT model introduced by Maki and Thompson in 1973 \cite{MR0366359}. In this model, only the spreader initiating a contact becomes a stifler when the contacted individual already knows the rumor. Starting with one spreader and all other individuals ignorant, Sudbury \cite{sudbury1985proportion} proved that the final proportion of individuals who never hear the rumor converges in probability to approximately $0.203$. Subsequent variations have extended these frameworks to complex social networks and introduced additional classes such as uninterested individuals or intermediate stages between awareness and dissemination, often analyzed through deterministic differential systems \cite{jie2025dynamic,lebensztayn2011limit,moreno2004dynamics,rada2021role,xia2023dynamic}.

While these iterations have expanded the model's scope, they remain largely within the Markovian paradigm. However, empirical reviews of misinformation psychology, notably Pennycook and Rand (2021), suggest that "truth discernment" is not an instantaneous event but a consequence of analytical reasoning \cite{pennycook2021psychology}. We posit that rumor spreading is not a memoryless process; instead, individuals who are aware of the rumor undergo a period of deliberation.

The deliberation period is motivated by the Two-Response Paradigm in cognitive psychology \cite{bago2020advancing}. This paradigm is rooted in dual-process theories of cognition, which distinguish `System 1' (fast, intuitive processes) and `System 2' (slower, more analytical processes) \cite{bago2020fake}. In this experimental framework, participants provide a first, immediate response to a stimulus, followed by a second response after being granted time to deliberate. Research demonstrates that while initial responses are often prone to the acceptance of misinformation, the second response -- facilitated by the engagement of analytical 'System 2' reasoning -- frequently leads to improved truth discernment. In our model, this period is represented by a random variable with distribution $F$. We assume that the probability of becoming a contestant is non-decreasing in the duration of deliberation.

The contestant class represents individuals who contest the rumor and may influence spreaders through contact, motivated by research on misinformation correction \cite{ecker2022psychological}. Spreader--stifler encounters already reduce the number of spreaders in the DK and MT models. Here this contact mechanism is combined with a deliberation period whose duration affects whether an individual becomes a spreader or a contestant. Conversion at every spreader--contestant contact and the absence of a return to spreading are simplifying assumptions of the model.

The aim of this paper is to propose a Non-Markovian rumor model on the complete graph and show a Functional Law of Large Numbers (FLLN) and a Functional Central Limit Theorem (FCLT) for this process. The results describe the limiting population proportions and their fluctuations, allowing the deliberation distribution to have atoms.

\subsection{Related Work on Non-Markovian Processes}

Reinert \cite{reinert1995asymptotic} proposed a non-Markovian generalization of Sellke's epidemic model \cite{sellke1983asymptotic}. Using the Stein method, she characterized the deterministic limit of the sequence of empirical measures.

Gao and Zhu \cite{gao2018functional} proved a Functional Central Limit Theorem for stationary Hawkes processes. Their proof uses a decomposition into independent copies and Hahn's theorem \cite{hahn1978central}, with moment estimates for the stationary processes. Here the contact intensities depend on the evolving population proportions.  

Pang and Pardoux \cite{Pardoux} established functional laws of large numbers and functional central limit theorems for non-Markovian epidemic models. For an overview, we refer to \cite{forien2021recent}. Our proofs adapt their comparison of stochastic intensities with the deterministic rates obtained from the FLLN, using the counting-process and compensation methods in \cite{Bremaud1981,bremaud2020point}. In the present model, the deliberation duration determines the decision time, and an additional mark selects the destination with a probability depending on that duration. Spreaders subsequently become contestants through contacts at rate $\lambda Y^nZ^n/n$, which depends on the evolving population. These transitions feed back into both contact intensities. The proofs therefore require simultaneous control of the two intensity errors and joint convergence of the delayed and contact fluctuations. The marked-sum approximation uses the same integral against the compensator for both sums, and the finite-dimensional Gaussian limits follow from Poisson superposition and the multivariate central limit theorem.

\subsection{Organization of the Paper} The paper is organized as follows. In Section \ref{Model}, we formally define the model and state the main results. We give the proof of the main results in Section \ref{Proofs}. Subsection \ref{preliminaries} is used to prove structural results for the class of counting processes we use. In subsection \ref{secFLLN}, we show the FLLN; and subsection \ref{secFCLT} is devoted to the proof of the FCLT.

\subsection{Notation} 
For the rest of this paper, $\N$ denotes the set of positive integers, $\N_0=\N\cup\{0\}$ and $\R_+$ the set of non-negative real numbers. The indicator function is denoted by $\1$. For $x=(x_1,\dots,x_k)\in\R^k$ and a function $f:[0,T]\to\R^k$, write
\begin{equation*}
    |x|=\sum_{i=1}^k|x_i|,
    \qquad
    \Vert f\Vert_T=\sup_{0\leq t\leq T}|f(t)|.
\end{equation*}
We denote by $D([0,T],\R^k)$ the space of right-continuous functions with left limits (càdlàg functions), equipped with the Skorohod $J_1$ topology. We also write $D^k=D([0,\infty),\R^k)$ and $D=D^1$. Convergence in distribution is denoted by $\Rightarrow$. We say that $V^n$ converges to $V$ uniformly on compact intervals in probability if, for every $T,\varepsilon>0$,
\begin{equation*}
    \prob(\Vert V^n-V\Vert_T>\varepsilon)\longrightarrow0.
\end{equation*}
In particular, $\Vert f\Vert_1$ denotes the supremum norm on $[0,1]$. For integrable functions on $[0,T]$, we write $\Vert f\Vert_{L^1}=\int_0^T|f(s)|ds$. For a sequence of processes $(V^n)$, define the fluid-scaled process by $\bar V^n=n^{-1}V^n$. For a square-integrable martingale $V$, write $[V]$ for its quadratic variation. We refer to \cite{Bill,Kurtz} for the path-space definitions.

\section{Model and main results}\label{Model}
\subsection{Definition of the Model}

Consider a closed, homogeneously mixed population with $n$ individuals and let the dynamics be as follows: an individual in the population can be in any of the four classes: ignorant, aware, spreader or contestant. Denote by $X^n(t)$, $W^n(t)$, $Y^n(t)$, and $Z^n(t)$ the ignorant, aware, spreader, and contestant individuals at time $t\geq 0$, respectively. 

Let $\lambda\geq0$, $F$ be a distribution function on $\R_+$ with $F(0)=0$ and $G:\R_+\to[0,1]$ be a non-decreasing function. Each unordered pair of distinct individuals has an independent Poisson contact process of rate $\lambda/n$. Thus each individual meets a uniformly chosen member of the remaining population at total rate $\lambda(n-1)/n$. We take the contact frequency to be independent of the individuals' classes; their classes determine the outcome of an encounter. An ignorant individual meets a spreader and becomes aware. After an $F$-distributed deliberation period $\eta_i$, the $i$-th newly aware individual takes a side: with probability $G(\eta_i)$ becomes contestant, otherwise becomes spreader. A spreader individual meets a contestant one and is convinced to be a contestant.

\begin{figure}[H]
    \centering
    \begin{tikzpicture}[node distance=2.5cm, every node/.style={circle, minimum size=1.5em, align=center}]
    \node[draw, minimum height=1.5em, minimum width=3em] (Xn) {$X^n$};
    \node[draw, minimum height=1.5em, minimum width=3em, right=of Xn] (Wn) {$W^n$};
    \node[draw, minimum height=1.5em, minimum width=3em, above right=of Wn] (Yn) {$Y^n$};
    \node[draw, minimum height=1.5em, minimum width=3em, below right=of Wn] (Zn) {$Z^n$};

    \draw[->] (Xn) -- (Wn) node[midway, above] {$\tau_i^n$};
    \draw[->] (Wn)+(0.58,0) -- (Yn) node[pos=0.2, below] {$\eta_i$};
    \draw[->] (Wn)+(0.58,0) -- (Zn);
    \draw[->] (Yn) -- (Zn) node[midway, right] {$\xi_i^n$};
\end{tikzpicture}
    \caption{Transitions between population classes. The variables $\tau_i^n$ and $\xi_i^n$ denote contact epochs, while $\eta_i$ denotes a deliberation period.}
    \label{dynamics}
\end{figure}

The initially aware individuals begin their deliberation at time zero. Let $(\eta_i)$ and $(\eta_i^0)$ be two sequences of i.i.d. random variables with distribution $F$, and let $(\U_i)$ and $(\U_i^0)$ be two sequences of i.i.d. uniform random variables on $[0,1]$. These four sequences are mutually independent and independent of the contact processes and the initial population counts. The initial population counts are also independent of the contact processes.

The dissemination of the rumor is driven by the contacts of spreader individuals with ignorant ones but is receded by the contacts of spreader individuals with contestant ones. Let $A^n(t)$ be the process that counts contacts between ignorants and spreaders up to time $t$ and $B^n(t)$ be the process that counts contacts between spreaders and contestants up to time $t$. There are $X^n(t-)Y^n(t-)$ pairs of the first type and $Y^n(t-)Z^n(t-)$ of the second, at time $t$. It follows that the contact construction gives the stochastic intensities
\begin{equation*}
    \frac{\lambda}{n}X^n(t-)Y^n(t-) \quad\text{and}\quad \frac{\lambda}{n}Y^n(t-)Z^n(t-).
\end{equation*}

respectively, with respect to the completed, right-continuous natural filtration
\begin{equation*}
    \F_t^n=\bigcap_{r>t}\Bigl(
 \sigma\{X^n(s),W^n(s),Y^n(s),Z^n(s),A^n(s),B^n(s):0\le s\le r\}
 \vee\mathcal N^n\Bigr),
\end{equation*}
where $\mathcal{N}^{n}$ consists of the subsets of null events in the underlying probability space. Denote by $\tau_i^n$ and $\xi_i^n$ the $i$-th epochs of $A^n$ and $B^n$, respectively. The deliberation of the $i$-th aware individual ends at $\tau_i^{n}+\eta_i$.

The deliberation mechanism retains information about elapsed time. For a duration $\eta$ with distribution $F$, an individual still aware after time $r$, with $F(r)<1$, satisfies
\begin{equation*}
    \prob(\eta>r+h\mid\eta>r)
    =\frac{1-F(r+h)}{1-F(r)},\qquad h\geq0.
\end{equation*}
In general, this probability depends on $r$, so the population counts alone do not determine the future evolution. The model also includes Markovian special cases, such as exponential deliberation with a
constant decision probability $G$.

Since the population is closed,
\begin{equation*}
    X^n(t)=n-W^n(t)-Y^n(t)-Z^n(t).
\end{equation*}

The aware individuals consist of those initially aware whose deliberation has not ended and those who became aware after time $0$ and are still deliberating:
\begin{equation*}
    W^n(t)=\sum_{i=1}^{W^n(0)}\1(t<\eta_i^0)+\sum_{i=1}^{A^n(t)}\1(t<\tau_i^n+\eta_i).
\end{equation*}
The corresponding number of spreaders and contestants are
\begin{align*}
    Y^n(t)=&Y^n(0)+\sum_{j=1}^{W^n(0)}\1(\eta_j^0\leq t,\U^0_j> G(\eta_j^0))\\
    &+\sum_{i=1}^{A^n(t)}\1(\tau_i^n+\eta_i\leq t,\U_i> G(\eta_i))-B^n(t)
\end{align*}
and
\begin{align*}
    Z^n(t)=&Z^n(0)+\sum_{j=1}^{W^n(0)}\1(\eta_j^0\leq t,\U^0_j\leq G(\eta_j^0))\\&+\sum_{i=1}^{A^n(t)}\1(\tau_i^n+\eta_i\leq t,\U_i\leq G(\eta_i))+B^{n}(t),
\end{align*}
respectively.

\subsection{Limit proportions}\label{Results}
We investigate the limit behavior as the population size tends to infinity. We first state the assumption on the initial proportions.

\noindent {\bf{Condition I}} As $n \rightarrow + \infty$,
\begin{equation*}
    (\bar{W}^n(0),\bar{Y}^n(0),\bar{Z}^n(0))\to(\bar{W}(0),\bar{Y}(0),\bar{Z}(0))
\end{equation*} 
in probability. Here $\bar{W}(0),\bar{Y}(0)$ and $\bar{Z}(0)$ are positive constants such that $\bar{W}(0)+\bar{Y}(0)+\bar{Z}(0)<1$. As usual, $\bar{X}(0):=1-\bar{W}(0)-\bar{Y}(0)-\bar{Z}(0)$. 

For $t\geq0$, write
\begin{equation*}
    F^c(t)=1-F(t),\qquad
 \beta(t)=\int_0^tG(s)F(ds),\qquad
 H(t)=F(t)-\beta(t).
\end{equation*}

\begin{theorem}[FLLN]\label{FLLN}
Assume that Condition I holds. Then
\begin{equation*}
    \lim_{n\to\infty}(\bar{X}^n,\bar{W}^n,\bar{Y}^n,\bar{Z}^n)=(\bar{X},\bar{W},\bar{Y},\bar{Z})\quad\text{in }D^4,
\end{equation*}
in probability, with convergence uniform on every compact time interval, where $(\bar{X},\bar{W},\bar{Y},\bar{Z})$ is the unique solution in the population simplex of the system of deterministic equations
\begin{align}
&\bar{X}(t)=1-\bar{W}(t)-\bar{Y}(t)-\bar{Z}(t),
\label{Xbar} 
 \\
&\bar{W}(t)=\bar{W}(0)F^c(t)+\lambda\int_0^tF^c(t-s)\bar{X}(s)\bar{Y}(s)ds,
\label{Wbar} 
 \\
 &\begin{aligned}\label{Ybar} 
\bar{Y}(t)=&\,\bar{Y}(0)+H(t)\bar{W}(0)+\lambda\int_0^t H(t-s)\bar{X}(s)\bar{Y}(s)ds-\lambda\int_0^t\bar{Y}(s)\bar{Z}(s)ds,
\end{aligned}
 \\
&\begin{aligned}\label{Zbar}
\bar{Z}(t)=&\,\bar{Z}(0)+\bar{W}(0)\beta(t)+\lambda\int_0^t\beta(t-s)\bar{X}(s)\bar{Y}(s)ds+\lambda\int_0^t\bar{Y}(s)\bar{Z}(s)ds.
\end{aligned}
\end{align}
\end{theorem}

\begin{remark}[Deterministic deliberation]\label{example}
Let $d>0$ and suppose that every deliberation period lasts exactly $d$, so that $F(t)=\1(t\geq d)$. Put $g=G(d)$ and write $(t-d)^+=\max\{t-d,0\}$. Then $H(t)=(1-g)\1(t\geq d)$ and $\beta(t)=g\1(t\geq d)$. The equations in Theorem \ref{FLLN} become
\begin{align*}
\bar W(t)&=\bar W(0)\1(t<d)
 +\lambda\int_{(t-d)^+}^t\bar X(s)\bar Y(s)ds,\\
\bar Y(t)&=\bar Y(0)
 +(1-g)\left[\bar W(0)\1(t\geq d)
 +\lambda\int_0^{(t-d)^+}\bar X(s)\bar Y(s)ds\right]
 -\lambda\int_0^t\bar Y(s)\bar Z(s)ds,\\
\bar Z(t)&=\bar Z(0)
 +g\left[\bar W(0)\1(t\geq d)
 +\lambda\int_0^{(t-d)^+}\bar X(s)\bar Y(s)ds\right]
 +\lambda\int_0^t\bar Y(s)\bar Z(s)ds,
\end{align*}
where $\bar X=1-\bar W-\bar Y-\bar Z$. No individual completes deliberation before time $d$. At time $d$, the initially aware individuals give the jumps
\begin{equation*}
\begin{aligned}
\bar W(d)-\bar W(d-)&=-\bar W(0),\\
\bar Y(d)-\bar Y(d-)&=(1-g)\bar W(0),\\
\bar Z(d)-\bar Z(d-)&=g\bar W(0).
\end{aligned}
\end{equation*}
Before time $d$, no individual completes deliberation, and
\begin{equation*}
    \bar Y'(t)=-\lambda\bar Y(t)\bar Z(t),\qquad 0<t<d.
\end{equation*}
Under Condition I, this derivative is strictly negative if $\lambda>0$. The jump at $d$ is positive when $g<1$. For almost every $t>d$,
\begin{equation*}
    \bar Y'(t)=\lambda(1-g)\bar X(t-d)\bar Y(t-d)
               -\lambda\bar Y(t)\bar Z(t).
\end{equation*}
The subsequent evolution is determined by the balance between delayed arrivals into the spreader class and conversions through contact.
\end{remark}

Remark \ref{example} describes an idealized setting in which individuals complete a verification period of fixed duration $d$ before spreading or contesting the rumor. An exponential deliberation period with the same mean $d$ permits completion arbitrarily soon after awareness, whereas the fixed duration excludes completion before $d$ time units have elapsed. During this waiting period, existing contestants can still reduce the current spreader population through contact. An aware individual of age $r<d$ has $d-r$ time units of deliberation remaining, so population counts alone do not, in general, determine future completion times.

\subsection{Fluctuations Around the Limiting Proportions}
Beyond the macroscopic limit established in Theorem \ref{FLLN}, we investigate the system's fluctuations via the diffusion-scaled processes. Let
\begin{equation*}
    \hat{W}^n(t):=\sqrt{n}(\bar{W}^n(t)-\bar{W}(t)),\ \hat{Y}^n(t):=\sqrt{n}(\bar{Y}^n(t)-\bar{Y}(t)),
\end{equation*}
and
\begin{equation*}
    \hat{Z}^n(t):=\sqrt{n}(\bar{Z}^n(t)-\bar{Z}(t)).
\end{equation*}
Also let $\hat{X}^n(t):=-\hat{W}^n(t)-\hat{Y}^n(t)-\hat{Z}^n(t).$

\noindent {\bf{Condition II}} There exist random variables $\hat{W}(0)$, $\hat{Y}(0)$ and $\hat{Z}(0)$ such that, as $n$ goes to infinity,
\begin{equation*}
    (\hat{W}^n(0),\hat{Y}^n(0),\hat{Z}^n(0))\Rightarrow(\hat{W}(0),\hat{Y}(0),\hat{Z}(0))
\end{equation*}
and $\sup_{n}\E[\hat{W}^n(0)^2]+\sup_{n}\E[\hat{Y}^n(0)^2]+\sup_{n}\E[\hat{Z}^n(0)^2]<\infty$.

\begin{theorem}[FCLT]\label{FCLT}
Assume Condition II holds. Then, as $n \rightarrow \infty$,
\begin{equation*} 
    (\hat{X}^n,\hat{W}^n,\hat{Y}^n,\hat{Z}^n)\Rightarrow(\hat{X},\hat{W},\hat{Y},\hat{Z})
\end{equation*}
in $D([0,T],\R^4)$ with the Skorohod $J_1$ topology, for every $T>0$. The limit system $(\hat{X},\hat{W},\hat{Y},\hat{Z})$ is the unique càdlàg solution of
\begin{align}
    \hat{X}(t)=&-\hat{W}(t)-\hat{Y}(t)-\hat{Z}(t),\label{Xhat}\\
    \hat{W}(t)=&\hat{W}(0)F^c(t)+\hat{W}_{0}(t)+\hat{W}_1(t)\\
    &+\lambda\int_0^tF^c(t-s)(\hat{X}(s)\bar{Y}(s)+\bar{X}(s)\hat{Y}(s))ds,\nonumber\\
    \hat{Y}(t)=&\hat{W}(0)H(t)+\hat{Y}(0)+\hat{Y}_{0}(t)+\hat{Y}_1(t)-\hat{B}(t)\\
    &+\lambda\int_0^tH(t-s)(\hat{X}(s)\bar{Y}(s)+\bar{X}(s)\hat{Y}(s))ds\nonumber\\
    &-\lambda\int_0^t(\hat{Y}(s)\bar{Z}(s)+\bar{Y}(s)\hat{Z}(s))ds,\nonumber\\
    \hat{Z}(t)=&\hat{W}(0)\beta(t)+\hat{Z}(0)+\hat{Z}_{0}(t)+\hat{Z}_1(t)+\hat{B}(t)\label{Zhat}\\
    &+\lambda\int_0^t \beta(t-s)(\hat{X}(s)\bar{Y}(s)+\bar{X}(s)\hat{Y}(s))ds\nonumber\\
    &+\lambda\int_0^t(\hat{Y}(s)\bar{Z}(s)+\bar{Y}(s)\hat{Z}(s))ds.\nonumber
\end{align}
The seven driving processes $\hat{W}_{0},\hat{W}_{1},\hat{Y}_{0},\hat{Y}_{1},\hat{Z}_{0},\hat{Z}_{1}$ and $\hat{B}$ are jointly centred Gaussian, with covariances given in Table \ref{TableOfCovariances}, and are independent of $(\hat{W}(0),\hat{Y}(0),\hat{Z}(0))$. The processes $(\hat{W}_{0},\hat{Y}_{0},\hat{Z}_{0})$, $(\hat{W}_{1},\hat{Y}_{1},\hat{Z}_{1})$ and $\hat{B}$ are mutually independent. The processes $\hat{W}_{1},\hat{Y}_{1},\hat{Z}_{1}$ and $\hat{B}$ have continuous paths. The initial-mark processes $\hat{W}_{0},\hat{Y}_{0}$ and $\hat{Z}_{0}$ have càdlàg paths, and satisfy $\hat{W}_{0}+\hat{Y}_{0}+\hat{Z}_{0}=0$. If $F$ is continuous, the four limiting fluctuation processes are continuous. If the initial fluctuation vector is Gaussian, the limiting fluctuation process is Gaussian as well.
\end{theorem}

\begin{lemma}\label{covariances}
For $r,t\geq0$, the nonzero covariances of the driving processes in Theorem \ref{FCLT} are listed in Table \ref{TableOfCovariances}. All entries not listed are zero or follow by covariance symmetry.
    \begin{table}[H]
    \centering
    \begin{tabular}{|c|c|}
    \hline $\Cov[\hat{W}_{0}(t),\hat{W}_{0}(r)]$&$\bar{W}(0)F(r\wedge t)F^c(r\vee t)$\\
    \hline$\Cov[\hat{Y}_{0}(t),\hat{Y}_{0}(r)]$&$\bar{W}(0)(H(r\wedge t)-H(t)H(r))$\\
    \hline$\Cov[\hat{Z}_{0}(t),\hat{Z}_{0}(r)]$&$\bar{W}(0)(\beta(r\wedge t)-\beta(r)\beta(t))$\\
    \hline$\Cov[\hat{W}_{0}(t),\hat{Y}_{0}(r)]$&$\bar{W}(0)(F(t)H(r)-H(t\wedge r))$\\
    \hline$\Cov[\hat{W}_{0}(t),\hat{Z}_{0}(r)]$&$\bar{W}(0)(F(t)\beta(r)-\beta(t\wedge r))$\\
    \hline$\Cov[\hat{Y}_{0}(t),\hat{Z}_{0}(r)]$&$-\bar{W}(0)H(t)\beta(r)$\\
    \hline$\Cov[\hat{W}_{1}(t),\hat{W}_{1}(r)]$&$\lambda\int_0^{r\wedge t}F^c(r\vee t-s)\bar{X}(s)\bar{Y}(s)ds$\\
    \hline$\Cov[\hat{Y}_{1}(t),\hat{Y}_{1}(r)]$&$\lambda\int_0^{t\wedge r} H(t\wedge r-s)\bar{X}(s)\bar{Y}(s)ds$\\
    \hline$\Cov[\hat{Z}_{1}(t),\hat{Z}_{1}(r)]$&$\lambda\int_0^{t\wedge r}\beta(t\wedge r-s)\bar{X}(s)\bar{Y}(s)ds$\\
    \hline$\Cov[\hat{W}_{1}(t),\hat{Y}_{1}(r)]$&$\lambda\1(t<r)\int_0^{t\wedge r} (H(r-s)-H(t-s))\bar{X}(s)\bar{Y}(s)ds$\\
    \hline$\Cov[\hat{W}_{1}(t),\hat{Z}_{1}(r)]$&$\lambda\1(t<r)\int_0^{t\wedge r}(\beta(r-s)-\beta(t-s))\bar{X}(s)\bar{Y}(s)ds$\\
    \hline$\Cov[\hat{B}(t),\hat{B}(r)]$&$\lambda\int_0^{t\wedge r}\bar{Y}(s)\bar{Z}(s)ds$\\
    \hline
    \end{tabular}
    \caption{Covariances between the processes obtained in Theorem \ref{FCLT}. All unlisted covariances are zero or follow from $\Cov(U(t),V(r))=\Cov(V(r),U(t))$.}
    \label{TableOfCovariances}
\end{table}
\end{lemma}

\section{Proofs}\label{Proofs}

\subsection{Preliminaries}\label{preliminaries}
In this section, we establish the estimates for counting processes used in the proof of the Functional Law of Large Numbers. A sequence of processes is called C-tight if it is tight in the Skorohod topology and every subsequential limit has continuous paths.

\begin{lemma}\label{tightness}
    Let $\{C^n\}_{n\geq1}$ be a sequence of simple counting processes with $C^n(0)=0$ and predictable $\F^n$-stochastic intensity $\gamma_n$. Suppose that for each $T>0$, there is a deterministic constant $K_T$ such that
    \begin{equation*}
        0\leq\gamma_{n}(s)\leq nK_T,\quad 0\leq s\leq T,
    \end{equation*}
    almost surely, for all $n$. Then the sequence $\{\bar{C}^n\}_{n\geq1}$ is $C$-tight in $D$.
\end{lemma}
\begin{proof}
    Fix $T>0$ and consider $K=K_T$. Since
    \begin{equation*}
        \E\left[\int_0^{T}\frac{\gamma_n(s)}{n^2}ds\right]\leq\frac{K\cdot T}{n}<\infty,
    \end{equation*}
    by Theorem 5.1.33 in \cite{bremaud2020point}, the process
    \begin{equation*}
        L^{n}(t)=\bar{C}^{n}(t)-\frac{1}{n}\int_0^{t}\gamma_n(s)ds
    \end{equation*}
    is a square-integrable martingale, with
    \begin{equation}\label{squareLn}
        \E[L^n(t)^2]=\frac{1}{n^2}\E\left[\int_0^t\gamma_n(s)ds\right]\leq\frac{K\cdot T}{n},\quad t\leq T.
    \end{equation}
    In particular, for $a>0$,
    \begin{equation*}
        \prob(\Vert \bar{C}^n\Vert_{T}\geq a)=\prob(\bar{C}^n(T)\geq a)\leq\frac{\E[\bar{C}^n(T)]}{a}\leq\frac{K\cdot T}{a}.
    \end{equation*}
    
    Let $\tau$ be an $\F^n$-stopping time with $\tau\leq T$, and let $\sigma=(\tau+\delta)\wedge T$, where $\delta>0$. Since $L^n$ is a square-integrable martingale, optional sampling gives
    \begin{equation*}
        \E[L^n(\sigma)-L^n(\tau)]=0.
    \end{equation*}
    Hence, by Markov's inequality,
    \begin{align*}
        \prob\big(|\bar C^n(\sigma)-\bar C^n(\tau)|\geq\varepsilon\big)&\leq\frac{\E[\bar C^n(\sigma)-\bar C^n(\tau)]}{\varepsilon}\\
        &=\frac{1}{n\varepsilon}\E\left[\int_\tau^\sigma\gamma_n(s)\,ds\right]\\
        &\leq\frac{K\delta}{\varepsilon}.
    \end{align*}
    These two bounds prove tightness by Aldous' criterion (Theorem $16.10$ in \cite{Bill}). Moreover, since $C^n$ is simple,
    \begin{equation*}
        \sup_{0<t\leq T}|\bar{C}^{n}(t)-\bar{C}^{n}(t-)|\leq\frac{1}{n}.
    \end{equation*}
    Thus every subsequential limit has continuous paths by Theorem $13.4$ in \cite{Bill}. Since $T>0$ is arbitrary, the result follows.
\end{proof}

\begin{corollary}\label{AnBnTight}
    The sequences of counting processes $\{\bar{A}^n\}_{n\geq1}$ and $\{\bar{B}^n\}_{n\geq1}$ are C-tight on $D$.
\end{corollary}
\begin{proof}
    Since the population proportions take values in $[0,1]$, the stochastic intensities of $A^n$ and $B^n$ are bounded by $n\lambda$. The result follows from Lemma \ref{tightness}.
\end{proof}

For the next two results, the hypotheses of Lemma \ref{tightness} are assumed to hold for a filtration that also contains the marks already drawn. Let $\tau_i^n$ denote the epochs of $C^n$, with $\tau_i^n=\infty$ if the $i$-th epoch does not occur. At each epoch, a mark $a_i^n$ with law $\mu$ on a measurable space $E$ is sampled independently of the information available before that mark is drawn, including the event time and all previous marks. The marked counting measure
\begin{equation*}
 \mathcal C^n(ds,da)
 =\sum_{i:\tau_i^n<\infty}\delta_{(\tau_i^n,a_i^n)}(ds,da)
\end{equation*}
has predictable compensator
\begin{equation*}
 \nu^n(ds,da)=\gamma_n(s)\,ds\,\mu(da).
\end{equation*}
Observe that the intensity may depend on previous marks.

\begin{corollary}\label{sumtight}
    Assume the hypothesis of Lemma \ref{tightness} and the counting measure $\mathcal{C}^n$ as above. Let $h:\R_+\times E\to[0,1]$ be jointly  measurable and, for $\mu$-almost every $a$, non-decreasing and càdlàg in its first argument. Set
    \begin{align*}
        k(u)&=\int_{E}h(u,a)\mu(da),\\
        V^n(t)&=\frac{1}{n}\sum_{i=1}^{C^n(t)}h(t-\tau_i^n,a_i^n)=\frac{1}{n}\int_{(0,t]\times E}h(t-s,a)\mathcal{C}^n(ds,da),\\
        R^n(t)&=\frac{1}{n}\int_{0}^{t}k(t-s)\ \gamma_n(s)ds=\frac{1}{n}\int_{(0,t]\times E}h(t-s,a)\,\nu^n(ds,da).
    \end{align*}
    Then, for every $T>0$,
    \begin{equation}\label{VnapproxRn}
        \lim_{n\to\infty}\E[\Vert V^n-R^n\Vert_T^2]=0,
    \end{equation}
     and $\{V^n\}_{n\geq1}$ is $C$-tight on $D$. Both conclusions also hold for $h=\sum_{\ell=1}^{q}b_{\ell}h_{\ell}$, where $q<\infty$, the coefficients $b_{\ell}\in\R$ do not depend on $n$ and each $h_\ell$ satisfies the preceding conditions.
\end{corollary}
\begin{proof}
    Fix $T>0$ and consider $K=K_T$. Note that $k$ is non-decreasing, càdlàg and satisfies $0\leq k\leq1$. Let $0\leq t\leq t+v\leq T$,
    \begin{align*}
        0\leq R^{n}(t+v)-R^{n}(t)=&\frac{1}{n}\int_0^{t}(k(t+v-s)-k(t-s))\ \gamma_n(s)ds\\
        &+\frac{1}{n}\int_{t}^{t+v}k(t+v-s)\ \gamma_n(s)ds\\
        \leq&\ K\left(\int_0^{t}(k(t+v-s)-k(t-s))ds+\int_{t}^{t+v}k(t+v-s)ds\right)\\
        =&K\int_t^{t+v}k(u)du\\
        \leq&Kv
    \end{align*}
    Consequently, $\Vert R^{n}\Vert_{T}\leq K\cdot T$ and $R^{n}$ is $K$-Lipschitz. Since
    \begin{equation*}
        V^n(t)-R^n(t)=\frac1n\int_{(0,t]\times E}h(t-s,a)(\mathcal C^n-\nu^n)(ds,da),
    \end{equation*}
    for a fixed $t$, applying Theorem 5.1.33 in \cite{bremaud2020point} with integrand $n^{-1}h(t-s,a)\1(s\leq t)$,
    \begin{align*}
        \E[(V^{n}(t)-R^{n}(t))^{2}]&=\frac{1}{n^2}\E\left[\int_0^{t}\gamma_n(s)\int_Eh(t-s,a)^2\mu(da)ds\right]\\
        &\leq\frac{1}{n^2}\E\left[\int_0^t\gamma_n(s)ds\right]\\
        &\leq \frac{Kt}{n}.
    \end{align*}
    Now let $0=t_0<t_1<\dots<t_m=T$ be a deterministic partition with mesh at most $\delta$. For $t_{i}\leq t\leq t_{i+1}$, the monotonicity of $V^n$ and the bound just shown for $R^n$ give
    \begin{align*}
        V^{n}(t)-R^{n}(t)&\leq V^{n}(t_{i+1})-R^{n}(t_{i+1})+K\delta\\
        V^{n}(t)-R^{n}(t)&\geq V^{n}(t_{i})-R^{n}(t_{i})-K\delta.
    \end{align*}
    It follows that
    \begin{equation*}
        \Vert V^{n}-R^{n}\Vert_{T}\leq\max_{1\leq i\leq m}|V^{n}(t_{i})-R^{n}(t_{i})|+K\delta,
    \end{equation*}
    and
    \begin{align*}
        \E[\Vert V^n-R^n\Vert_{T}^2]&\leq2\sum_{i=1}^{m}\E[|V^{n}(t_{i})-R^{n}(t_{i})|^2]+2K^2\delta^2\\
        &\leq \frac{2(m+1)KT}{n}+2K^2\delta^2.
    \end{align*}
    For each fixed $\delta>0$, choose such a partition independently of $n$. Then
    \begin{equation*}
        0\leq\limsup_{n\to\infty}\E[\Vert V^n-R^n\Vert_T^2]\leq2K^2\delta^2.
    \end{equation*}
    Letting $\delta\downarrow0$, we obtain the first assertion. In particular, for every $\varepsilon>0$,
    \begin{equation*}
        \prob(\Vert V^n-R^n\Vert_T>\varepsilon)\leq\varepsilon^{-2}\E[\Vert V^n-R^n\Vert_T^{2}]\to0.
    \end{equation*}
    The bounds $\Vert R^n\Vert_T\leq K\cdot T$ and $|R^n(t)-R^n(s)|\leq K|t-s|$ verify Aldous' criterion as in Lemma \ref{tightness}. By the approximation in equation \eqref{VnapproxRn}, the same conclusion holds for $V^n$. The last part follows by repeating the arguments above with $h_\ell$ in place of $h$.
\end{proof}

 \begin{theorem}\label{VnEVnzero}
    Under the conditions of Corollary \ref{sumtight}, let
    \begin{equation*}
        \tilde{V}^n(t):=\frac{1}{n}\sum_{i=1}^{C^n(t)}\E[h(t-\tau_i^n,a_i^n)|\tau_i^n],
    \end{equation*}
    with the convention that the conditional expectation is zero on the events $\{t<\tau_{i}^{n}\}$. Then for every $T>0$,
    \begin{equation*}
        \lim_{n\to\infty}\E[\Vert V^n-\tilde{V}^n\Vert_T^2]=0.
    \end{equation*}
    In particular, $V^n-\tilde{V}^n$ converges to zero uniformly on compact intervals in probability, and $\{\tilde{V}^{n}\}_{n\geq1}$ is C-tight.
\end{theorem}

\begin{proof}
    First suppose that $h$ is non-decreasing, càdlàg in its first argument and takes values in $[0,1]$. Since the mark $a_i^n$ is sampled independently at time $\tau_i^n$, 
    \begin{equation*}
        \E[h(t-\tau_i^n,a_i^n)|\tau_i^{n}]=k(t-\tau_i^n),\quad\text{on }\{t\geq\tau_i^n\}.
    \end{equation*}
    Consequently,
    \begin{equation*}
        \tilde{V}^n(t)=\frac{1}{n}\sum_{i=1}^{C^n(t)}k(t-\tau_i^n)=\frac{1}{n}\int_{(0,t]\times E}k(t-s)\mathcal{C}^n(ds,da).
    \end{equation*}
    Apply Corollary \ref{sumtight} first to $h(u,a)$ and then to $k(u)$. In the second application, note that
    \begin{equation*}
        \frac{1}{n}\int_0^t\int_Ek(t-s)\mu(da)\gamma_n(s)ds=\frac{1}{n}\int_0^tk(t-s)\gamma_n(s)ds=R^n(t).
    \end{equation*}
    Thus the same process $R^n$ approximates both sums. Also,
    \begin{equation*}
        \E[\Vert V^n-\tilde{V}^n\Vert_T^2]\leq2\E[\Vert V^n-R^n\Vert_T^2]+2\E[\Vert R^n-\tilde{V}^n\Vert_T^2]\to0.
    \end{equation*}
    The corollary also gives C-tightness of $\{\tilde{V}^n\}_{n\geq1}$. If $h=\sum_{\ell=1}^{q}b_\ell h_\ell$, linearity and the preceding conclusion give
    \begin{equation*}
        \E[\Vert V^n-\tilde{V}^n\Vert_T^{2}]\leq q\sum_{\ell=1}^{q}b_\ell^2\E[\Vert V_{\ell}^n-\tilde{V}_{\ell}^n\Vert_T^{2}]\to0
    \end{equation*}
    and the finite-sum part of Corollary \ref{sumtight} gives C-tightness.
\end{proof}

\subsection{Functional Law of Large Numbers}\label{secFLLN}

    In this subsection, we prove Theorem \ref{FLLN}. We first consider the individuals initially aware and then those who become aware after time zero. For a mark $m=(\eta,u)\in E:=\R_+\times[0,1]$ and $s\geq0$, define
    \begin{align*}
        h_W(s,m)&=\1(s<\eta)\\
        h_Y(s,m)&=\1(\eta\leq s, u>G(\eta))\\
        h_Z(s,m)&=\1(\eta\leq s, u\leq G(\eta)).
    \end{align*}
    Let $\mu(dm)=F(d\eta)\,du$ on $E$. For $j\in\{W,Y,Z\}$, put
    \begin{equation*}
        k_j(s)=\int_E h_j(s,m)\mu(dm).
    \end{equation*}
    It holds that
    \begin{equation*}
        k_W(s)=F^c(s),\quad k_Y(s)=H(s)\quad\text{and}\quad k_Z(s)=\beta(s),
    \end{equation*}
    The functions $h_Y$ and $h_Z$ are non-decreasing and càdlàg in $s$ and $h_W=1-h_Y-h_Z$. Moreover, $F(0)=0$ gives $h_Y(0,m)=h_Z(0,m)=0$ for $\mu$-almost every $m$.

    Write $m_i^0=(\eta_i^0,\U_i^0)$ and $m_i=(\eta_i,\U_i)$. The contributions of the initially aware individuals are
    \begin{equation*}
        \bar{W}_0^n(t)=\frac{1}{n}\sum_{i=1}^{W^{n}(0)}h_W(t,m_i^0),\quad\bar{Y}_0^n(t)=\frac{1}{n}\sum_{i=1}^{W^{n}(0)}h_Y(t,m_i^0)\quad\text{and}\quad\bar{Z}_0^n(t)=\frac{1}{n}\sum_{i=1}^{W^{n}(0)}h_Z(t,m_i^0).
    \end{equation*}
    Also, denote the probability of becoming a spreader at the end of the deliberation by
    \begin{equation*}
        p=\prob(\mathcal{U}>G(\eta))=\int_{[0,\infty)}(1-G(r))F(dr)=\lim_{t\to\infty}H(t),
    \end{equation*}
    $\mathcal U$ uniform on $[0,1]$ and $\eta$ distributed according to $F$. In particular, $0\leq H(t)\leq p$ and $0\leq\beta(t)\leq1-p$. Let $U_1,U_2,\dots$ be independent uniform random variables on $[0,1]$, independent of the initial counts. We may represent the three indicator processes of the $i$-th initially aware individual, jointly in law, by
    \begin{equation*}
        \big(1-\1(U_i\leq H(t))-\1(p<U_i\leq p+\beta(t)),\quad \1(U_i\leq H(t)),\quad \1(p<U_i\leq p+\beta(t))\big),\qquad t\geq0.
    \end{equation*}

    Indeed, $U_i\in[0,p]$ assigns the destination of $Y$ at the end of deliberation, whereas $U_i\in(p,1]$ assigns the destination $Z$. For $0\leq s<t$, the probabilities of completing deliberation during $(s,t]$ with these two destinations are
    \begin{equation*}
        H(t)-H(s)\quad\text{and}\quad\beta(t)-\beta(s),
    \end{equation*}
    respectively, which are the lengths of $(H(s),H(t)]$ and $(p+\beta(s),p+\beta(t)]$. These intervals determine the joint law at any finite collection of times. Thus they also determine the law of the càdlàg indicator processes. The variables $U_i$ are auxiliary and need not coincide with the original marks $\mathcal{U}_i^0$. For $m\geq1$, put
    \begin{equation}\label{alphamu}
        \alpha_m(u)=\frac{1}{\sqrt{m}}\sum_{i=1}^m(\1(U_i\leq u)-u),\qquad u\in[0,1].
    \end{equation}
    Set $\alpha_0=0$. By Donsker's theorem (Theorem $14.3$ in \cite{Bill}),
    \begin{equation}\label{alphaBrownian}
        \alpha_m\Rightarrow\B\qquad\text{in}\ D([0,1]),
    \end{equation}
    where $\B$ denotes a standard Brownian bridge: a centred Gaussian process with continuous paths and covariance
    \begin{equation*}
        \E[\B(u)\B(v)]=u\wedge v-uv,\quad u,v\in[0,1].
    \end{equation*}
    In particular, the sequence $\{\Vert\alpha_m\Vert_1\}_{m\geq1}$ is tight.
    
    Let $w_0=\bar{W}(0)$. 

    \begin{lemma}\label{W0initialconv}
        Under Condition I, $(\bar{W}_0^n,\bar{Y}_0^n,\bar{Z}_0^n)$ converges, uniformly on compact intervals in probability, to
        \begin{equation*}
            (w_0F^c,w_0H,w_0\beta).
        \end{equation*}
    \end{lemma}
    \begin{proof}
        We first prove the assertion for $\bar Y_0^n$. In the uniform representation above, let
        \begin{equation*}
            m_n=\lfloor nw_0\rfloor,\qquad\check{Y}_0^n(t)=\frac{1}{n}\sum_{i=1}^{m_n}\1(U_i\leq H(t)).
        \end{equation*}
        Since each summand belongs to $[0,1]$,
        \begin{equation*}
            \Vert\bar Y_0^n-\check Y_0^n\Vert_T\leq\frac{|W^n(0)-m_n|}{n}\leq|\bar W^n(0)-w_0|+\frac{1}{n}\to0,
        \end{equation*}
        in probability. Moreover,
        \begin{equation*}
            \check Y_0^n(t)=\frac{m_n}{n}H(t)+\frac{\sqrt{m_n}}{n}\alpha_{m_n}(H(t)).
        \end{equation*}
        Therefore,
        \begin{equation*}
            \Vert\check Y_0^n-w_0H\Vert_T\leq\left|\frac{m_n}{n}-w_0\right|+\frac{\sqrt{m_n}}{n}\Vert\alpha_{m_n}\Vert_1=O(n^{-1})+O_{\prob}(n^{-1/2})\to0
        \end{equation*}
        in probability. This proves the assertion for $Y$, the case for $Z$ being analogous, and for $W$ it follows by recalling that $1-H-\beta=F^c$.
    \end{proof}

    For the individuals who become aware after time zero, set
    \begin{align*}
        \bar W_1^n(t)=\frac{1}{n}\sum_{i=1}^{A^n(t)}h_W(t-\tau_i^n,m_i),\qquad&\widetilde W_1^n(t)=\int_{(0,t]}F^c(t-s)d\bar A^n(s),\\
        \bar Y_1^n(t)=\frac{1}{n}\sum_{i=1}^{A^n(t)}h_Y(t-\tau_i^n,m_i),\qquad&\widetilde Y_1^n(t)=\int_{(0,t]}H(t-s)d\bar A^n(s),\\
        \bar Z_1^n(t)=\frac{1}{n}\sum_{i=1}^{A^n(t)}h_Z(t-\tau_i^n,m_i),\qquad&\widetilde Z_1^n(t)=\int_{(0,t]}\beta(t-s)d\bar A^n(s).
    \end{align*}

    \begin{lemma}\label{Ansumconv}
        For every $T>0$,
        \begin{equation*}
            \Vert\bar{W}^n_1-\widetilde{W}^n_1\Vert_T+\Vert\bar{Y}^n_1-\widetilde{Y}^n_1\Vert_T+\Vert\bar{Z}^n_1-\widetilde{Z}^n_1\Vert_T\to0,
        \end{equation*}
        in probability as $n\to\infty$.
    \end{lemma}
    \begin{proof}
        The intensity of $A^n$ is bounded by $n\lambda$. The functions $h_Y$ and $h_Z$ are non-decreasing and càdlàg, so Theorem \ref{VnEVnzero} gives
        \begin{equation*}
            \Vert\bar{Y}^n_1-\widetilde{Y}^n_1\Vert_T+\Vert\bar{Z}^n_1-\widetilde{Z}^n_1\Vert_T\to0
        \end{equation*}
        in probability. Since $h_W=1-h_Y-h_Z$,
        \begin{equation*}
            \bar W_1^n=\bar A^n-\bar Y_1^n-\bar Z_1^n,\qquad \widetilde W_1^n=\bar A^n-\widetilde Y_1^n-\widetilde Z_1^n.
        \end{equation*}
        Subtracting proves the assertion for $W$.
    \end{proof}

    \begin{lemma}\label{AnBn}
        For every $T>0$,
    \begin{equation*}
        \left\Vert\bar{A}^n-\lambda\int_0^.\bar{X}^n(s)\bar{Y}^n(s)ds\right\Vert_T\to0
        \quad\text{and}\quad
        \left\Vert\bar{B}^n-\lambda\int_0^.\bar{Y}^n(s)\bar{Z}^n(s)ds\right\Vert_T\to0,
    \end{equation*}
        in probability, as $n\to\infty$.
    \end{lemma}
    \begin{proof}
        We prove the result for $A^n$. Its predictable intensity is
        \begin{equation*}
            n\lambda\bar X^n(s-)\bar Y^n(s-)
        \end{equation*}
        It is bounded by $n\lambda$. Applying equation \eqref{squareLn} with $C^n=A^n$ and then Doob's $L^2$ inequality, obtained by integration and Cauchy-Schwarz from Theorem A.3.4 in \cite{bremaud2020point}, we obtain
        \begin{equation}\label{DoobFLLN}
            \E\left[\left\Vert\bar A^n-\lambda\int_0^.\bar X^n(s)\bar Y^n(s)ds\right\Vert_T^2\right]\leq\frac{4\lambda T}{n}.
        \end{equation}
        The left limit in the intensity may be replaced by the value of the process inside the integral: each càdlàg path differs from its left limit at only countably many times, a set of zero Lebesgue measure $(ds)$. Markov's inequality proves uniform convergence in probability. The proof for $B^n$ is analogous.
    \end{proof}

    \begin{proof}[Theorem \ref{FLLN}]
        Fix $T>0$. By Lemma \ref{W0initialconv},
        \begin{equation*}
            (\bar W_0^n,\bar Y_0^n,\bar Z_0^n)\to w_0(F^c,H,\beta)
        \end{equation*}
        uniformly on $[0,T]$ in probability. Corollary \ref{AnBnTight} gives C-tightness of $(\bar A^n,\bar B^n)$. From any subsequence we may therefore extract a further subsequence such that $(\bar A^{n_k},\bar B^{n_k})\Rightarrow(\bar A,\bar B)$, where both limits have continuous paths. The compensators of $\bar A^n$ and $\bar B^n$ have non-decreasing $\lambda$-Lipschitz paths starting at zero. Lemma \ref{AnBn} shows that their differences from the counting processes tend uniformly to zero in probability. Consequently, $\bar A$ and $\bar B$ have these same path properties.

        Include in this convergence the initial counts divided by $n$, the errors in Lemmas \ref{W0initialconv} and \ref{Ansumconv}, and the compensator errors in Lemma \ref{AnBn}. Each error has zero limit uniformly in probability. Apply the Skorohod representation theorem (Theorem $6.7$ in \cite{Bill}) in the finite product of the relevant $D$ spaces with the $J_1$ topology and the Euclidean spaces of initial values, including the error processes themselves. It provides a probability space on which the joint convergence holds almost surely. Since the error limits are the zero paths, their convergence on this representation is uniform. In particular,
        \begin{equation*}
            \Vert\bar A^n-\bar A\Vert_T+\Vert\bar B^n-\bar B\Vert_T\to0,\qquad a.s.,
        \end{equation*}
        since their limits are continuous.

        We next pass to the integrals defining $\widetilde{W}_1^n,\widetilde{Y}_1^n$ and $\widetilde{Z}_1^n$. For càdlàg functions $k,c_n,c$ of bounded variation, with $c_n(0)=c(0)=0$, write
        \begin{equation*}
            k(t-s)=k(0)+\int_{(0,t-s]}dk(u).
        \end{equation*}
        Substituting the preceding identity, we obtain
    \begin{align*}
        \int_{(0,t]}k(t-s)\,d(c_n-c)(s)&=k(0)(c_n-c)(t)\\
        &\quad+\int_{(0,t]}\int_{(0,t-s]}dk(u)\,d(c_n-c)(s).
    \end{align*}
    By Fubini's theorem,
    \begin{align*}
        \int_{(0,t]}\int_{(0,t-s]}dk(u)\,d(c_n-c)(s)&=\int_{(0,t]}\int_{(0,t-u]}d(c_n-c)(s)\,dk(u)\\
        &=\int_{(0,t]}(c_n-c)(t-u)\,dk(u),
    \end{align*}
    where the last equality follows from $c_n(0)=c(0)=0$. It follows that
        \begin{equation}\label{supkdcnc}
            \sup_{t\leq T}\left|\int_{(0,t]}k(t-s)d(c_n-c)(s)\right|\leq(|k(0)|+\Vert k\Vert^{TV}_{(0,T]})\cdot\Vert c_n-c\Vert_T,
        \end{equation}
        where $\Vert k\Vert^{TV}_{(0,T]}$ is the total variation of $k$ on $[0,T]$. Apply equation \eqref{supkdcnc} with $c_n=\bar A^n$, $c=\bar A$ and $k=F^c,H,\beta$. The constants on the right are at most $2,1,1$, respectively. Lemmas \ref{W0initialconv} and \ref{Ansumconv} now give uniform convergence, on the representation, to
        \begin{align*}
            \bar X(t)&=1-\bar W(t)-\bar Y(t)-\bar Z(t),\\
            \bar W(t)&=w_0F^c(t)+\int_{(0,t]}F^c(t-s)d\bar A(s),\\
            \bar Y(t)&=\bar Y(0)+w_0H(t)+\int_{(0,t]}H(t-s)d\bar A(s)-\bar B(t),\\
            \bar Z(t)&=\bar Z(0)+w_0\beta(t)+\int_{(0,t]}\beta(t-s)d\bar A(s)+\bar B(t).
        \end{align*}
        
        Note that, for any $x_1,x_2,y_1,y_2\in[0,1]$,
        \begin{equation}\label{x12y12}
            |x_1y_1-x_2y_2|=|(x_1-x_2)y_1+x_2(y_1-y_2)|\leq|x_1-x_2|+|y_1-y_2|.
        \end{equation}
        Hence the uniform convergence of the proportions gives
        \begin{equation*}
            \sup_{t\leq T}\left|\int_0^t\bar X^n(s)\bar Y^n(s)ds-\int_0^t\bar X(s)\bar Y(s)ds\right|\leq T(\Vert\bar X^n-\bar X\Vert_T+\Vert\bar Y^n-\bar Y\Vert_T)\to0
        \end{equation*}
        and
        \begin{equation*}
            \sup_{t\leq T}\left|\int_0^t\bar Y^n(s)\bar Z^n(s)ds-\int_0^t\bar Y(s)\bar Z(s)ds\right|\leq T(\Vert\bar Y^n-\bar Y\Vert_T+\Vert\bar Z^n-\bar Z\Vert_T)\to0.
        \end{equation*}
        Together with Lemma \ref{AnBn}, this proves
        \begin{equation*}
            \bar A(t)=\lambda\int_0^t\bar X(s)\bar Y(s)ds
            \quad\text{and}\quad
            \bar B(t)=\lambda\int_0^t\bar Y(s)\bar Z(s)ds.
        \end{equation*}
        The preceding equations thus reduce to \eqref{Xbar}--\eqref{Zbar}. Since the population proportions take values in the closed simplex $\{(x,w,y,z)\in[0,1]^4:x+w+y+z=1\}$, their limits do as well. For uniqueness, we let $(x_1,w_1,y_1,z_1)$ and $(x_2,w_2,y_2,z_2)$ be two solutions in the population simplex with the same initial values. Let
        \begin{equation*}
            d(t)=\Vert w_1-w_2\Vert_t+\Vert y_1-y_2\Vert_t+\Vert z_1-z_2\Vert_t.
        \end{equation*}
        Conservation of population and \eqref{x12y12} imply, for every $s\leq T$,
        \begin{align*}
            |x_1(s)-x_2(s)|&\leq d(s),\\
            |x_1(s)y_1(s)-x_2(s)y_2(s)|&\leq 2d(s),\\
            |y_1(s)z_1(s)-y_2(s)z_2(s)|&\leq d(s).
        \end{align*}
        Since $F^c,H,\beta\in[0,1]$, subtracting the respective integral equations yields
        \begin{align*}
            \Vert w_1-w_2\Vert_t&\leq2\lambda\int_0^td(s)ds,\\
            \Vert y_1-y_2\Vert_t&\leq(2+1)\lambda\int_0^td(s)ds,\\
            \Vert z_1-z_2\Vert_t&\leq(2+1)\lambda\int_0^td(s)ds.
        \end{align*}
        Adding the terms, we obtain
        \begin{equation*}
            d(t)\leq8\lambda\int_0^td(s)ds.
        \end{equation*}
        Gronwall's inequality proves that $d(t)=0$ for every $t\leq T$.

        Thus the solution is deterministic and the same solution is obtained from every subsequence. To conclude in the uniform norm, suppose that, for some $\varepsilon,\rho>0$, a further subsequence satisfied
        \begin{equation*}
            \prob\left(\max_{J\in\{X,W,Y,Z\}}\Vert\bar J^{n_k}-\bar J\Vert_T>\varepsilon\right)\geq\rho.
        \end{equation*}

        Apply the preceding steps to that subsequence. Estimate \eqref{supkdcnc} together with the vanishing errors already included in the representation gives
        \begin{equation*}
            \max_{J\in\{X,W,Y,Z\}}\Vert\bar J^{n_k}-\bar J\Vert_T\to0\quad a.s.
        \end{equation*}
        The representation preserves the laws of these variables. Their probabilities of exceeding $\varepsilon$ must therefore tend to zero, a contradiction. This proves uniform convergence in probability on $[0,T]$. Since $T>0$ is arbitrary, the result follows.
    \end{proof}

\subsection{Functional Central Limit Theorem}\label{secFCLT}

We now prove Theorem \ref{FCLT}. Recall the functions $h_j$ and $k_j$, $j\in\{W,Y,Z\}$, defined in the proof of the FLLN. Following \cite{Pardoux}, we compare the stochastic intensities with the deterministic rates obtained in that theorem.

Recall $E=\R_+\times[0,1]$ and $\mu(dm)=F(d\eta)du$, where $m=(\eta,u)$. The contact construction has an equivalent representation by two independent Poisson random measures
\begin{equation*}
    Q_A^n(ds,dm,dv)\quad\text{and}\quad Q_B^n(ds,dv),
\end{equation*}
with mean measures $ds\mu(dm)dv$ and $dsdv$, respectively. They are independent of the initial counts and marks. The variable $s$ is a contact time, $m$ is the mark of an individual becoming aware, and $v$ is used to obtain the required contact rate. For the stochastic integrals below, let
\begin{equation*}
    \begin{aligned}
        \mathcal{H}_t^n=\sigma\langle (W^n(0),Y^n(0),Z^n(0)), (m_i^0)_{i\geq1},Q_A^n((0,s]\times C\times D),Q_B^n((0,s]\times D):\\
    0\leq s\leq t, C\in\mathcal{B}(E), D\in\mathcal{B}(\R_+)\text{ bounded}\rangle.
    \end{aligned}
\end{equation*}
Then set
\begin{equation*}
    \G_t^n=\bigcap_{r>t}(\mathcal{H}_r^n\vee\mathcal{N}^n),
\end{equation*}
where $\mathcal{N}^n$ is the set of null events used in definition of $\F_t^n$. Thus $\G_t^n$ is complete and right-continuous. The initial marks are known at time zero, and the Poisson random measures are observed only up to the current time. Put
\begin{align*}
    a_n(s)&=\lambda \bar X^n(s-)\bar Y^n(s-),\qquad &b_n(s)&=\lambda \bar Y^n(s-)\bar Z^n(s-),\\
    a(s)&=\lambda \bar X(s-)\bar Y(s-),\qquad &b(s)&=\lambda \bar Y(s-)\bar Z(s-).
\end{align*}

At $s=0$, left limits mean the initial values. Then the contact counts are represented by
\begin{align}
    A^n(t)&=\int_{(0,t]\times E\times\R_+}\1(v\leq na_n(s))Q_A^n(ds,dm,dv),\label{Anmeasure}\\
    B^n(t)&=\int_{(0,t]\times\R_+}\1(v\leq nb_n(s))Q_B^n(ds,dv).\label{Bnmeasure}
\end{align}

The first indicator does not depend on $m$, so each newly aware individual receives a mark with law $\mu$, independently of the information preceding the contact. The construction therefore gives the contact intensities and mark distribution specified in the model. Since $a_n,b_n\leq\lambda$, only finitely many relevant Poisson points occur on each bounded time interval. The construction is non-explosive. Write the compensated measures
\begin{equation*}
    \widetilde Q_A^n=Q_A^n-ds\mu(dm)dv\quad\text{and}\quad\widetilde Q_B^n=Q_B^n-dsdv.
\end{equation*}
For $j\in\{W,Y,Z\}$, define the initial contributions $(\ell=0)$ and the contributions from individuals becoming aware after time zero $(\ell=1)$ by
\begin{align}
    \hat V^n_{j,0}(t)&=\frac1{\sqrt n}\sum_{i=1}^{W^n(0)}\left(h_j(t,m_i^0)-k_j(t)\right),\label{Vj0def}\\
    \hat V^n_{j,1}(t)&=\frac1{\sqrt n}\int_{(0,t]\times E\times\R_+}h_j(t-s,m)\1(v\leq na_n(s))\widetilde Q_A^n(ds,dm,dv).\label{Vj1def}
\end{align}
Also define 
\begin{equation}\label{Bhatdef}
    \hat{B}^n(t)=\frac{1}{\sqrt{n}}\int_{(0,t]\times\R_+}\1(v\leq nb_n(s))\widetilde Q_B^n(ds,dv)=\frac{B^{n}(t)-n\int_0^tb_n(s)ds}{\sqrt{n}}.
\end{equation}

Observe that for every fixed $t\geq0$, the integrand in equation \eqref{Vj1def} is predictable in $s$ and measurable in the mark variables $(m,v)$: the factor $h_j(t-s,m)$ is deterministic and $a_n(s)$ depends on the state immediately before $s$. Including the factor $n^{-1/2}$, its squared integral with respect to the mean measure is at most $\lambda t$. Since $a_n,b_n\leq\lambda$, the Poisson measures may be restricted to $v\leq n\lambda$, which contains the support of every integrand and has finite mean measure on $[0,T]$. Thus the compensated integrals are square-integrable (Theorem $5.1.33$ in \cite{bremaud2020point}). This does not make $\hat{V}_{j,1}^n$ a martingale as $t$ varies, since $h_j(t-s,m)$ also depends on $t$. Note that
\begin{equation*}
    \bar X^n\bar Y^n-\bar X\bar Y=\bar Y^n(\bar X^n-\bar X)+\bar X(\bar Y^n-\bar Y).
\end{equation*}
Applying the same identity to $\bar Y^n\bar Z^n-\bar Y\bar Z$, we obtain
\begin{align*}
    \Delta_A^n(s)=\sqrt{n}(a_n(s)-a(s))=\lambda(\bar Y^n(s-)\hat X^n(s-)+\bar X(s-)\hat Y^n(s-))\\
    \Delta_B^n(s)=\sqrt{n}(b_n(s)-b(s))=\lambda(\bar Z^n(s-)\hat Y^n(s-)+\bar Y(s-)\hat Z^n(s-)).
\end{align*}
The population and fluctuation paths are càdlàg, so their values and left limits differ only at countably many times. Replacing left limits therefore does not change any integral with respect to $ds$. The fluctuation processes therefore satisfy
\begin{equation}\label{fluctuation-equations}
\begin{aligned}
    \hat W^n(t)&=\hat W^n(0)k_W(t)+\hat V_{W,0}^n(t)+\hat V_{W,1}^n(t)
    +\int_0^t k_W(t-s)\Delta_A^n(s)ds,\\
    \hat Y^n(t)&=\hat Y^n(0)+\hat W^n(0)k_Y(t)+\hat V_{Y,0}^n(t)+\hat V_{Y,1}^n(t)-\hat B^n(t)\\
    &\quad+\int_0^t k_Y(t-s)\Delta_A^n(s)ds-\int_0^t\Delta_B^n(s)ds,\\
    \hat Z^n(t)&=\hat Z^n(0)+\hat W^n(0)k_Z(t)+\hat V_{Z,0}^n(t)+\hat V_{Z,1}^n(t)+\hat B^n(t)\\
    &\quad+\int_0^t k_Z(t-s)\Delta_A^n(s)ds+\int_0^t\Delta_B^n(s)ds.
\end{aligned}
\end{equation}
with $\hat{X}^n=-\hat{W}^n-\hat{Y}^n-\hat{Z}^n$.

\begin{lemma}\label{LemmaConv0} Under Condition II, for every $T>0$,
\begin{equation*}
    (\hat{V}^n_{W,0},\hat{V}^n_{Y,0},\hat{V}^n_{Z,0})\Rightarrow(\hat{W}_{0},\hat{Y}_{0},\hat{Z}_{0})
\end{equation*}
as $n\to\infty$ in $D([0,T],\R^3)$. The process $(\hat{W}_{0},\hat{Y}_{0},\hat{Z}_{0})$ is centred Gaussian with the covariances in Table \ref{TableOfCovariances} and is independent of the initial fluctuation vector.
\end{lemma}
\begin{proof}
    Conditioning on the initial counts and applying Donsker's theorem for the empirical process (Theorem 14.3 in \cite{Bill}), the result follows from the representation used in the proof of the FLLN. The same argument gives joint convergence with the initial fluctuations in Condition II, with independent limits.
\end{proof}

\begin{lemma}\label{controlofvariances}
    Under Condition II, for every $T>0$,
    \begin{equation*}
        \sup_{n\geq1}\sup_{t\leq T}\E[(|\hat X^n(t)|+|\hat W^n(t)|+|\hat Y^n(t)|+|\hat Z^n(t)|)^2]<\infty.
    \end{equation*}
\end{lemma}
\begin{proof}
    Throughout the proof, $C_T$ denotes a finite deterministic constant, independent of $n$. Set
    \begin{align*}
        U_n(t)&=|\hat X^n(t)|+|\hat W^n(t)|+|\hat Y^n(t)|+|\hat Z^n(t)|,\\
        I_n&=|\hat W^n(0)|+|\hat Y^n(0)|+|\hat Z^n(0)|,\\
        N_n(t)&=2|\hat B^n(t)|+\sum_{j\in \{W,Y,Z\}}(|\hat V_{j,0}^n(t)|+|\hat V_{j,1}^n(t)|).
    \end{align*}
    Conditionally on $W^n(0)$, the centred indicators in equation \eqref{Vj0def} are independent, with mean zero and variance $k_j(t)(1-k_j(t))$. Hence,
    \begin{equation*}
        \E[\hat V_{j,0}^{n}(t)^2|W^n(0)]=\frac{W^n(0)}{n}k_j(t)(1-k_j(t)).
    \end{equation*}
    Consequently, $\E[\hat V_{j,0}^n(t)^2]\leq1$. Applying the second-moment identity for compensated Poisson integrals (Theorem 5.1.33 in \cite{bremaud2020point}) to \eqref{Vj1def} gives
    \begin{align*}
        \E[\hat V_{j,1}^n(t)^2]&=\frac{1}{n}\E\left[\int_0^t\int_E\int_{0}^{\infty}h_j(t-s,m)^2\1(v\leq na_n(s))dv\mu(dm)ds\right]\\
        &=\E\left[\int_0^ta_n(s)\left(\int_Eh_j(t-s,m)\mu(dm)\right)ds\right]\\
        &=\int_0^tk_j(t-s)\E[a_n(s)]ds\\
        &\leq\lambda T.
    \end{align*}
    Here we used $h_{j}^2=h_{j}$ and the definition of $k_j$. Similarly,
    \begin{equation*}
        \E[\hat B^n(t)^2]=\int_0^t\E[b_n(s)]ds\leq\lambda T.
    \end{equation*}
    Consequently, Condition II and Cauchy-Schwarz give
    \begin{equation*}
        \sup_{n}\E[I_n^2]<\infty\quad\text{and}\quad\sup_n\sup_{t\leq T}\E[N_n(t)^2]\leq7(3+7\lambda T).
    \end{equation*}
    Here we used that $(x_1+\dots+x_q)^2\leq q(x_1^2+\dots+x_q^2)$ with $q=7$. To estimate the remaining terms in $U_n(t)$, we use $k_W+k_Y+k_Z=1$ and $k_j\geq0$. Taking absolute values in \eqref{fluctuation-equations} and adding gives
    \begin{equation*}
        |\hat W^n(t)|+|\hat Y^n(t)|+|\hat Z^n(t)|\leq I_n+N_n(t)+\int_0^t(|\Delta_A^n(s)|+2|\Delta_B^n(s)|)ds.
    \end{equation*}
    Moreover, since all population proportions belong to $[0,1]$,
    \begin{align*}
        |\Delta_A^n(s)|&\leq\lambda(|\hat X^n(s-)|+|\hat Y^n(s-)|),\\
        |\Delta_B^n(s)|&\leq\lambda(|\hat Y^n(s-)|+|\hat Z^n(s-)|).
    \end{align*}
    The identity $\hat X^n=-\hat W^n-\hat Y^n-\hat Z^n$ implies
    \begin{equation*}
        U_n(t)\leq2(|\hat W^n(t)|+|\hat Y^n(t)|+|\hat Z^n(t)|).
    \end{equation*}
    Combining the preceding inequalities,
    \begin{align*}
        U_n(t)&\leq2I_n+2N_n(t)+2\lambda\int_0^t(|\hat X^n(s)|+3|\hat Y^n(s)|+2|\hat Z^n(s)|)ds\\
        &\leq2I_n+2N_n(t)+6\lambda\int_0^tU_n(s)ds.
    \end{align*}
    Using again that $(x+y+z)^2\leq3(x^2+y^2+z^2)$ and applying Cauchy-Schwarz, we obtain
    \begin{align*}
        \E[U_n(t)^2]&\leq12\E[I_n^2]+12\E[N_n(t)^2]+108\lambda^2t\int_0^t\E[U_n^2(s)]ds\\
        &\leq C_T+C_T\int_0^t\E[U_n(s)^2]ds.
    \end{align*}
    Gronwall's inequality proves the assertion.
\end{proof}

We next replace $a_n,b_n$ in \eqref{Vj1def} and \eqref{Bhatdef} by $a,b$. Using the same Poisson random measures, define the modified fluctuations
\begin{align*}
 \widetilde V_{j,1}^n(t)&=\frac1{\sqrt n}
       \int_{(0,t]\times E\times\R_+}
        h_j(t-s,m)\1(v\leq na(s))\,
                         \widetilde Q_A^n(ds,dm,dv),
 \qquad j\in\{W,Y,Z\}.\\
 \widetilde M_A^n(t)&=\frac1{\sqrt n}
       \int_{(0,t]\times E\times\R_+}
              \1(v\leq na(s))\,\widetilde Q_A^n(ds,dm,dv),\\
 \widetilde B^n(t)&=\frac1{\sqrt n}
       \int_{(0,t]\times\R_+}
              \1(v\leq nb(s))\,\widetilde Q_B^n(ds,dv).
\end{align*}
Since $h_W+h_Y+h_Z=1$,
\begin{equation*}
 \widetilde V_{W,1}^n=\widetilde M_A^n-
                     \widetilde V_{Y,1}^n-\widetilde V_{Z,1}^n.
\end{equation*}

\begin{lemma}\label{comparison-intensities}
Under Condition II, for every $T>0$,
\begin{equation*}
 \sum_{j\in\{W,Y,Z\}}
       \Vert\hat V_{j,1}^n-\widetilde V_{j,1}^n\Vert_T
       +\Vert\hat B^n-\widetilde B^n\Vert_T\longrightarrow0
 \qquad\text{in probability}.
\end{equation*}
\end{lemma}
\begin{proof}
Fix $T>0$. Throughout the proof, $C_T$ denotes a finite deterministic constant, independent of $n$, whose value may change from line to line. Lemma~\ref{controlofvariances} also gives the moment bound at left limits, by Fatou's lemma. Thus, using the bounds for $\Delta_A^n$ and $\Delta_B^n$ in that proof,
\begin{equation}\label{pointwise-deltas}
 \sup_{n\geq1}\sup_{s\leq T}
 \E\left[|\Delta_A^n(s)|^2+|\Delta_B^n(s)|^2\right]\leq C_T.
\end{equation}
Since $\Delta_A^n=\sqrt n(a_n-a)$ and
$\Delta_B^n=\sqrt n(b_n-b)$, Cauchy--Schwarz gives
\begin{equation}\label{intensity-lone-bounds}
 \sup_{s\leq T}\bigl(\E[|a_n(s)-a(s)|]+\E[|b_n(s)-b(s)|]\bigr)
 \leq C_Tn^{-1/2}.
\end{equation}
For $j\in\{Y,Z\}$, let
\begin{equation}\label{Djnandtheta}
 D_j^n(t):=\hat V_{j,1}^n(t)-\widetilde V_{j,1}^n(t)
 \quad\text{and}\quad
 \theta_n(s,\rho)=\1(\rho\leq na_n(s))-\1(\rho\leq na(s)).
\end{equation}
Then
\begin{align*}
 |\theta_n(s,\rho)|&=\theta_n(s,\rho)^2=\1(n\min(a_n(s),a(s))<\rho\leq n\max(a_n(s),a(s))),\\
 \int_0^\infty\theta_n(s,\rho)^2\,d\rho & =n|a_n(s)-a(s)|.
\end{align*}
For each fixed $t\leq T$, the second-moment identity for compensated integrals (Theorem 5.1.33 in \cite{bremaud2020point}) gives
\begin{align*}
 \E[D_j^n(t)^2]
 &=\frac1n\E\left[\int_0^t\int_E\int_0^\infty
       h_j(t-s,m)^2\theta_n(s,\rho)^2\,d\rho\,\mu(dm)\,ds\right]\\
 &=\E\left[\int_0^t\int_E
       h_j(t-s,m)|a_n(s)-a(s)|\,\mu(dm)\,ds\right]\\
 &=\int_0^t k_j(t-s)\E[|a_n(s)-a(s)|]\,ds\\
 &\leq C_Tn^{-1/2}.
\end{align*}
The second equality uses $h_j^2=h_j$, and the third uses the definition of $k_j$. The last bound follows from
\eqref{intensity-lone-bounds}.

To pass from fixed times to the supremum on $[0,T]$, let
\begin{align*}
 P_j^n(t)&=\frac1{\sqrt n}
       \int_{(0,t]\times E\times\R_+}
       h_j(t-s,m)|\theta_n(s,\rho)|Q_A^n(ds,dm,d\rho),\\
 R_j^n(t)&=\int_0^t k_j(t-s)|\Delta_A^n(s)|\,ds.
\end{align*}
Both processes are non-decreasing, since $h_j$ and $k_j$ are non-decreasing in their first argument and the weights are nonnegative. For $0\leq u\leq v\leq T$, we define
\begin{equation*}
 d_{u,v}(s,m):=\1(s\leq v)h_j(v-s,m)-\1(s\leq u)h_j(u-s,m)\in\{0,1\}.
\end{equation*}
A term whose time indicator vanishes is defined to be zero. Its mean with respect to the mark is
\begin{equation}\label{quvs}
 q_{u,v}(s):=\int_Ed_{u,v}(s,m)\,\mu(dm)=\1(s\leq u)(k_j(v-s)-k_j(u-s))+\1(u<s\leq v)k_j(v-s).
\end{equation}
In particular,
\begin{align}
 0\leq\int_0^v q_{u,v}(s)\,ds
 &=\int_0^v k_j(v-s)\,ds-\int_0^u k_j(u-s)\,ds\notag\\
 &=\int_0^v k_j(r)\,dr-\int_0^u k_j(r)\,dr\notag\\
 &=\int_u^v k_j(r)\,dr\leq v-u.
 \label{completion-increment-mass}
\end{align}

By the triangle inequality and the definitions of $P_j^n$ and $R_j^n$,
\begin{equation}
 |D_j^n(v)-D_j^n(u)|
       \leq P_j^n(v)-P_j^n(u)+R_j^n(v)-R_j^n(u).
 \label{comparison-increment-bound}
\end{equation}
By Cauchy--Schwarz, \eqref{pointwise-deltas} and
\eqref{completion-increment-mass},
\begin{align*}
 \E[(R_j^n(v)-R_j^n(u))^2]
 &=\E\left[\left(\int_0^v q_{u,v}(s)|\Delta_A^n(s)|\,ds\right)^2\right]\\
 &\leq\left(\int_0^v q_{u,v}(s)\,ds\right)
       \int_0^v q_{u,v}(s)\E[\Delta_A^n(s)^2]\,ds\\
 &\leq C_T(v-u)^2.
\end{align*}
Moreover,
\begin{align*}
 P_j^n(v)-P_j^n(u)-(R_j^n(v)-R_j^n(u))=\frac1{\sqrt n}\int_{(0,v]\times E\times\R_+}
        d_{u,v}(s,m)|\theta_n(s,\rho)|
                                      \widetilde Q_A^n(ds,dm,d\rho).
\end{align*}
Since $d_{u,v}^2=d_{u,v}$, the same second-moment identity gives
\begin{align*}
 \E\Big[\big((P_j^n(v)-P_j^n(u))-(R_j^n(v)-R_j^n(u))\big)^2\Big]&=\int_0^v q_{u,v}(s)\E[|a_n(s)-a(s)|]\,ds\\
 &\qquad\leq C_T(v-u)n^{-1/2}.
\end{align*}
Consequently, $(x+y)^2\leq2x^2+2y^2$ gives
\begin{equation*}
 \E[(P_j^n(v)-P_j^n(u))^2]
       \leq C_T((v-u)n^{-1/2}+(v-u)^2).
\end{equation*}

Fix $\varepsilon>0$ and $0<\delta\leq T$. Set $r=\lceil T/\delta\rceil$ and $t_\ell=\min\{\ell\delta,T\}$ for $0\leq\ell\leq r$. Then $0=t_0<\cdots<t_r=T$. By \eqref{comparison-increment-bound}, for $t\in[t_\ell,t_{\ell+1}]$,
\begin{equation*}
 |D_j^n(t)|
 \leq|D_j^n(t_\ell)|+(P_j^n(t_{\ell+1})-P_j^n(t_\ell))+(R_j^n(t_{\ell+1})-R_j^n(t_\ell)).
\end{equation*}
If each of the three terms on the right is at most $\varepsilon$ on every interval, then $\Vert D_j^n\Vert_T\leq3\varepsilon$. A union bound and Chebyshev's inequality therefore give
\begin{align*}
    \prob(\Vert D_j^n\Vert_T>3\varepsilon)
    &\leq\sum_{\ell=0}^r\prob(|D_j^n(t_\ell)|>\varepsilon)
    +\sum_{\ell=0}^{r-1}\prob(P_j^n(t_{\ell+1})-P_j^n(t_\ell)>\varepsilon)\\
    &\quad+\sum_{\ell=0}^{r-1}\prob(R_j^n(t_{\ell+1})-R_j^n(t_\ell)>\varepsilon)\\
    &\leq\frac{C_T}{\varepsilon^2}\bigl((r+1)n^{-1/2}+Tn^{-1/2}+T\delta\bigr).
\end{align*}
Here we used
\begin{equation*}
    \sum_{\ell=0}^{r-1}(t_{\ell+1}-t_\ell)^2
    \leq\delta\sum_{\ell=0}^{r-1}(t_{\ell+1}-t_\ell)=T\delta.
\end{equation*}
Then, for every fixed $\delta$,
\begin{equation*}
 \limsup_{n\to\infty}\prob(\Vert D_j^n\Vert_T>3\varepsilon)
       \leq\frac{C_TT\delta}{\varepsilon^2}.
\end{equation*}
Letting $\delta\downarrow0$ proves convergence in probability for
$j\in\{Y,Z\}$.

For the total number of arrivals, let
\begin{equation*}
 \hat M_A^n(t):=\frac{A^n(t)-n\int_0^t a_n(s)\,ds}{\sqrt n}.
\end{equation*}
By \eqref{Anmeasure},
\begin{equation*}
    \hat M_A^n(t)-\widetilde M_A^n(t)
    =\frac1{\sqrt n}\int_{(0,t]\times E\times\R_+}
    \theta_n(s,\rho)\widetilde Q_A^n(ds,dm,d\rho),
\end{equation*}
where $\theta_n$ is defined in \eqref{Djnandtheta}. The integrand is predictable, and
\begin{equation*}
    \frac1n\E\left[\int_0^T\int_E\int_0^\infty
    \theta_n(s,\rho)^2d\rho\,\mu(dm)ds\right]
    =\int_0^T\E[|a_n(s)-a(s)|]ds<\infty.
\end{equation*}
Therefore, Theorem 5.1.33 in \cite{bremaud2020point} gives that $\hat M_A^n-\widetilde M_A^n$ is a square-integrable $\G^n$-martingale. Applying Doob's $L^2$ inequality as in \eqref{DoobFLLN} and then the second-moment identity, we obtain
\begin{align*}
    \E[\Vert\hat M_A^n-\widetilde M_A^n\Vert_T^2]
    &\leq4\E[|\hat M_A^n(T)-\widetilde M_A^n(T)|^2]\\
    &=4\int_0^T\E[|a_n(s)-a(s)|]ds\\
    &\leq C_Tn^{-1/2}.
\end{align*}
Since $h_W+h_Y+h_Z=1$,
\begin{equation*}
    \hat V_{W,1}^n-\widetilde V_{W,1}^n
    =(\hat M_A^n-\widetilde M_A^n)-D_Y^n-D_Z^n.
\end{equation*}
Thus the assertion also holds for $W$. Similarly,
\begin{equation*}
    \hat B^n(t)-\widetilde B^n(t)
    =\frac1{\sqrt n}\int_{(0,t]\times\R_+}
    \bigl(\1(\rho\leq nb_n(s))-\1(\rho\leq nb(s))\bigr)
    \widetilde Q_B^n(ds,d\rho).
\end{equation*}
The integrand is predictable and satisfies the same square-integrability condition with $b_n,b$ in place of $a_n,a$. Hence $\hat B^n-\widetilde B^n$ is also a square-integrable $\G^n$-martingale, and
\begin{equation*}
    \E[\Vert\hat B^n-\widetilde B^n\Vert_T^2]
    \leq4\int_0^T\E[|b_n(s)-b(s)|]ds
    \leq C_Tn^{-1/2}.
\end{equation*}
This concludes the proof.
\end{proof}

\begin{lemma}\label{LemmaConv}
Under Condition II, for every $T>0$,
\begin{equation*}
 (\hat V_{W,1}^n,\hat V_{Y,1}^n,\hat V_{Z,1}^n,\hat B^n)
 \Rightarrow(\hat W_1,\hat Y_1,\hat Z_1,\hat B)
 \quad\text{in }D([0,T],\R^4),
\end{equation*}
where the limit is a centred Gaussian process with continuous paths and the covariances given in Table~\ref{TableOfCovariances}. This convergence holds jointly with the initial fluctuations in Condition II and the three processes in Lemma~\ref{LemmaConv0}. The limiting process $(\hat W_1,\hat Y_1,\hat Z_1,\hat B)$ is independent of these initial terms.
\end{lemma}
\begin{proof}
By Lemma~\ref{comparison-intensities}, it is enough to prove the assertions for the processes defined with deterministic intensities. We first establish tightness, and then compute the joint finite-dimensional limits.

For $j\in\{Y,Z\}$, the process
\begin{equation*}
 \int_{(0,t]\times E\times\R_+}
 h_j(t-s,m)\1(\rho\leq na(s))Q_A^n(ds,dm,d\rho)
\end{equation*}
counts completions of deliberation with destination $j$. For example, its increment over $(u,v]$ counts precisely the points
satisfying
\begin{equation*}
 0<s\leq v,\qquad u<s+\eta\leq v\quad\text{and}\quad
 \rho\leq na(s),
\end{equation*}
with the decision mark choosing destination $j$. Disjoint completion intervals correspond to disjoint deterministic sets of $Q_A^n$, therefore their counts are independent Poisson variables. The cumulative mean is $nc_j(t)$, where
\begin{equation*}
 c_j(t)=\int_0^t k_j(t-s)a(s)\,ds.
\end{equation*}
The function $q_{u,v}$ defined in equation \eqref{quvs} gives
\begin{equation*}
 0\leq c_j(v)-c_j(u)
 =\int_0^v q_{u,v}(s)a(s)\,ds
 \leq\lambda(v-u),\qquad 0\leq u\leq v\leq T.
\end{equation*}
Hence $c_j$ is continuous and non-decreasing, with $c_j(0)=0$.

For $j\in\{Y,Z\}$, define
\begin{equation*}
    \mathcal H_t^{n,j}
    =\bigcap_{u>t}
    \left(\sigma\bigl(\widetilde V_{j,1}^n(s):0\leq s\leq u\bigr)\vee\mathcal N^n\right).
\end{equation*}
Since $c_j$ is deterministic, this is also the completed, right-continuous filtration generated by the corresponding completion count. Its independent increments give
\begin{equation*}
    \E\left[\widetilde V_{j,1}^n(t)-\widetilde V_{j,1}^n(s)\mid\mathcal H_s^{n,j}\right]=0,
    \qquad 0\leq s\leq t.
\end{equation*}
Thus $\widetilde V_{j,1}^n$ is a square-integrable martingale with respect to $(\mathcal H_t^{n,j})_{t\geq0}$. Since $c_j$ is continuous, the completion count has jumps of size one almost surely. Therefore,
\begin{align*}
 \sup_{t\leq T}|\Delta\widetilde V_{j,1}^n(t)|&\leq n^{-1/2},\\
 [\widetilde V_{j,1}^n](t)
 &=c_j(t)+n^{-1/2}\widetilde V_{j,1}^n(t).
\end{align*}
Since $\E[|\widetilde V_{j,1}^n(T)|^2]=c_j(T)$, Doob's inequality gives
\begin{equation*}
 \E[\Vert[\widetilde V_{j,1}^n]-c_j\Vert_T^2]\leq\frac{4c_j(T)}n
 \longrightarrow0.
\end{equation*}
Theorem 1.4 in \cite[Chapter 7]{Kurtz} therefore gives convergence to a continuous centred Gaussian process. In particular, $\{\widetilde V_{j,1}^n\}_n$ is $C$-tight. The same argument applies to $\widetilde M_A^n$ and $\widetilde B^n$, with
\begin{equation*}
 c_A(t)=\int_0^t a(s)\,ds,
 \qquad c_B(t)=\int_0^t b(s)\,ds.
\end{equation*}

The four sequences of processes, $\widetilde M_A^n,\widetilde V_{Y,1}^n, \widetilde V_{Z,1}^n,\widetilde B^n$, are tight in the product space. Every subsequential limit has continuous coordinates. On a joint Skorohod representation, coordinatewise convergence to such a limit is uniform, hence also convergence in the vector $J_1$ topology. Finally,
\begin{equation*}
 \widetilde V_{W,1}^n
 =\widetilde M_A^n-\widetilde V_{Y,1}^n-\widetilde V_{Z,1}^n.
\end{equation*}
The same uniform convergence applies to this linear combination. Thus $(\widetilde V_{W,1}^n,\widetilde V_{Y,1}^n,\widetilde V_{Z,1}^n,\widetilde B^n)$ is $C$-tight.

Fix an integer $d\geq1$ and times $t_1,\ldots,t_d\in[0,T]$. Define the finite measure
\begin{equation*}
 \nu(ds,dm)=a(s)\,ds\,\mu(dm),
 \qquad \nu((0,T]\times E)\leq\lambda T.
\end{equation*}
For a measurable set $C\subseteq(0,T]\times E$, put
\begin{equation*}
 S^n(C)=\int_{(0,T]\times E\times\R_+}
          \1((s,m)\in C)\1(\rho\leq na(s))Q_A^n(ds,dm,d\rho).
\end{equation*}
This is a Poisson random measure of intensity $n\nu$. To verify this directly, disjoint sets $C_1,\ldots,C_q$ give disjoint sets
\begin{equation*}
 \{(s,m,\rho):(s,m)\in C_i,\ 0<\rho\leq na(s)\},
 \qquad 1\leq i\leq q,
\end{equation*}
of the original Poisson random measure. Thus the variables $S^n(C_i)$ are independent and Poisson, with means
\begin{equation*}
 \int_{C_i}\int_0^{na(s)}d\rho\,ds\,\mu(dm)=n\nu(C_i).
\end{equation*}

Let $\Pi_1,\Pi_2,\ldots$ be independent Poisson random measures with intensity $\nu$, then,
\begin{equation*}
 S^n\ \stackrel{\mathrm{law}}=\ \sum_{i=1}^n\Pi_i.
\end{equation*}

For $j\in\{W,Y,Z\}$ and $1\leq k\leq d$, set
\begin{equation*}
 g_{j,k}(s,m)=
 \begin{cases}
 h_j(t_k-s,m),&0<s\leq t_k,\\
 0,&t_k<s\leq T.
 \end{cases}
\end{equation*}
For each $\ell\geq1$, let $\zeta_\ell$ be the vector whose coordinates are
\begin{equation*}
    \zeta_{\ell,j,k}
    =\int_{(0,T]\times E}g_{j,k}(s,m)(\Pi_\ell-\nu)(ds,dm).
\end{equation*}
These vectors are independent and identically distributed. Their means and covariances are
\begin{align*}
 \E[\zeta_{\ell,j,k}]&=0,\\
 \E[\zeta_{\ell,j,k}\zeta_{\ell,i,q}]
 &=\int_{(0,T]\times E}g_{j,k}(s,m)g_{i,q}(s,m)\,\nu(ds,dm).
\end{align*}
In particular,
\begin{equation*}
 \E[\zeta_{\ell,j,k}^2]
 =\int g_{j,k}(s,m)^2\,\nu(ds,dm)
 \leq\nu((0,T]\times E)\leq\lambda T.
\end{equation*} 
Moreover,
\begin{equation*}
 \big(\widetilde V_{j,1}^n(t_k):
      j\in\{W,Y,Z\},\ 1\leq k\leq d\big)
 \ \stackrel{\mathrm{law}}=\ \frac1{\sqrt n}\sum_{\ell=1}^n\zeta_\ell.
\end{equation*}
The multivariate central limit theorem (Theorem 3.10.7 in \cite{Durrett2019}) therefore gives a centred Gaussian limit in $\R^{3d}$. For classes $i,j\in\{W,Y,Z\}$ at times $t,r$, its covariance is
\begin{equation*}
 \int_0^{t\wedge r}a(s)\int_E
       h_i(t-s,m)h_j(r-s,m)\,\mu(dm)\,ds.
\end{equation*}
Together with $C$-tightness, this identifies the functional limit of $(\widetilde V_{W,1}^n,\widetilde V_{Y,1}^n,
\widetilde V_{Z,1}^n)$.

The process $\widetilde B^n$ depends only on $Q_B^n$, whereas the other three processes depend only on $Q_A^n$. They are independent for each $n$. Their joint limit therefore has independent components between these two groups, and
\begin{equation*}
 \Cov(\hat B(t),\hat B(r))=\int_0^{t\wedge r}b(s)\,ds.
\end{equation*}
The processes defined with deterministic intensities are also independent of the initial counts and marks. The conditioning argument in Lemma \ref{LemmaConv0} gives
\begin{equation*}
    (\hat W^n(0),\hat Y^n(0),\hat Z^n(0),\alpha_{W^n(0)})
    \Rightarrow
    (\hat W(0),\hat Y(0),\hat Z(0),\B).
\end{equation*}
Independence gives joint convergence with this vector. Lemma \ref{comparison-intensities} transfers this joint convergence to $(\hat V_{W,1}^n,\hat V_{Y,1}^n,\hat V_{Z,1}^n,\hat B^n)$. The initial empirical representation then gives joint convergence with the three processes in Lemma \ref{LemmaConv0}. The covariance entries are computed below.
\end{proof}

\begin{proof}[Lemma~\ref{covariances}] 
It remains to compute the entries of Table~\ref{TableOfCovariances}. Write $\mu[f]=\int_E f(m)\mu(dm)$. For $u,v\geq0$, the definitions of the indicators give
\begin{align*}
 \mu[h_W(u)h_W(v)]&=F^c(u\vee v),\\
 \mu[h_Y(u)h_Y(v)]&=H(u\wedge v),\\
 \mu[h_Z(u)h_Z(v)]&=\beta(u\wedge v),\\
 \mu[h_W(u)h_Y(v)]&=\1(u<v)(H(v)-H(u)),\\
 \mu[h_W(u)h_Z(v)]&=\1(u<v)(\beta(v)-\beta(u)),\\
 \mu[h_Y(u)h_Z(v)]&=0.
\end{align*}
For example, $h_W(u)h_Y(v)$ is the indicator of $\{u<\eta\leq v,\ U>G(\eta)\}$. For the initially aware individuals, the empirical-process covariance in Lemma~\ref{LemmaConv0} is
\begin{equation*}
 \bar W(0)\{\mu[h_i(t)h_j(r)]-k_i(t)k_j(r)\}.
\end{equation*}
In particular,
\begin{align*}
 \mu[h_W(t)h_Y(r)]-F^c(t)H(r)
 &=H(r)-H(t\wedge r)-F^c(t)H(r)\\
 &=F(t)H(r)-H(t\wedge r),\\
 \mu[h_Y(t)h_Z(r)]-H(t)\beta(r)&=-H(t)\beta(r).
\end{align*}
For individuals arriving after time zero, the covariance is instead the integral of $h_i(t-s,m)h_j(r-s,m)$ with respect to $a(s)\,ds\,\mu(dm)$. In this Poisson second-moment identity no product of means is subtracted. Substituting the six identities above proves the entries of Table~\ref{TableOfCovariances}. The other entries follow by symmetry or are zero. In particular, $\Cov(\hat Y_1(t),\hat Z_1(r))=0$ because $h_Y(u,m)h_Z(v,m)=0$ for every $u,v$.
\end{proof}

\begin{proof}[Theorem~\ref{FCLT}]
Fix $T>0$. The proofs of Lemmas \ref{LemmaConv0} and \ref{LemmaConv} give joint convergence of the initial fluctuations, the empirical process and the contributions after time zero. By the FLLN and Theorem 3.9 in \cite{Bill}, this convergence also holds jointly with $\bar W^n(0)$ and $(\bar Y^n,\bar Z^n)$, the latter considered in $L^1([0,T])^2$.

Applying the Skorohod representation theorem (Theorem 6.7 in \cite{Bill}) in
\begin{equation*}
    \R^4\times D([0,1])\times D([0,T],\R^4)\times L^1([0,T])^2,
\end{equation*}
with $J_1$ on the path spaces, we may assume that the convergences hold almost surely, preserving the joint laws. Since $\B,\hat W_1,\hat Y_1,\hat Z_1$ and $\hat B$ are continuous, their $J_1$ convergences are uniform \cite[Section 12]{Bill}. Hence, almost surely,
\begin{align*}
    (\hat W^n(0),\hat Y^n(0),\hat Z^n(0))
    &\longrightarrow(\hat W(0),\hat Y(0),\hat Z(0)),\\
    |\bar W^n(0)-w_0|+\Vert\alpha_{W^n(0)}-\B\Vert_1
    &\longrightarrow0,\\
    \Vert\hat V_{W,1}^n-\hat W_1\Vert_T
    +\Vert\hat V_{Y,1}^n-\hat Y_1\Vert_T
    +\Vert\hat V_{Z,1}^n-\hat Z_1\Vert_T
    +\Vert\hat B^n-\hat B\Vert_T
    &\longrightarrow0,\\
    \int_0^T\bigl(|\bar Y^n(s)-\bar Y(s)|+|\bar Z^n(s)-\bar Z(s)|\bigr)ds
    &\longrightarrow0.
\end{align*}

On this probability space, take
\begin{align*}
    \hat Y_0(t)&=\sqrt{w_0}\B(H(t)),\\
    \hat Z_0(t)&=\sqrt{w_0}\bigl(\B(p+\beta(t))-\B(p)\bigr),\\
    \hat W_0(t)&=-\hat Y_0(t)-\hat Z_0(t).
\end{align*}
These processes have the law given in Lemma \ref{LemmaConv0}. For functions $f,g$ on $[0,1]$,
\begin{align*}
 \Vert f\circ H-g\circ H\Vert_T&\leq\Vert f-g\Vert_1,\\
 \sup_{t\leq T}|f(p+\beta(t))-f(p)-g(p+\beta(t))+g(p)|
 &\leq2\Vert f-g\Vert_1.
\end{align*}
Apply these inequalities to $f=\alpha_{W^n(0)}$ and $g=\B$, and use $\sqrt{\bar W^n(0)}\to\sqrt{w_0}$. Since the three initial contributions sum to zero, this gives
\begin{equation*}
 \Vert\hat V_{W,0}^n-\hat W_0\Vert_T
 +\Vert\hat V_{Y,0}^n-\hat Y_0\Vert_T
 +\Vert\hat V_{Z,0}^n-\hat Z_0\Vert_T\longrightarrow0
 \qquad a.s.
\end{equation*}
This argument applies also when $H$ and $\beta$ have jumps.

Set
\begin{align*}
 c_W^n(t)&=\hat W^n(0)k_W(t)+\hat V_{W,0}^n(t)+\hat V_{W,1}^n(t),\\
 c_Y^n(t)&=\hat Y^n(0)+\hat W^n(0)k_Y(t)
           +\hat V_{Y,0}^n(t)+\hat V_{Y,1}^n(t)-\hat B^n(t),\\
 c_Z^n(t)&=\hat Z^n(0)+\hat W^n(0)k_Z(t)
           +\hat V_{Z,0}^n(t)+\hat V_{Z,1}^n(t)+\hat B^n(t).
\end{align*}
Let $c=(c_W,c_Y,c_Z)$ be given by the limiting quantities, replacing $(\hat V_{W,\ell}^n,\hat V_{Y,\ell}^n,\hat V_{Z,\ell}^n)$ by $(\hat W_\ell,\hat Y_\ell,\hat Z_\ell)$, for $\ell\in\{0,1\}$. The preceding convergences and $0\leq k_j\leq1$ imply $\Vert c^n-c\Vert_T\to0$ almost surely.

Apply Lemma \ref{Volterraintegral} with $L=1$ and fixed coefficients $\bar X,\bar Y$. Set
\begin{equation*}
    \phi^n=\Phi(c^n,\bar Y^n,\bar Z^n),
    \qquad
    \phi=\Phi(c,\bar Y,\bar Z).
\end{equation*}
By \eqref{fluctuation-equations} and the uniqueness and measurability in Lemma \ref{Volterraintegral}, $\phi^n$ has the same law as $(\hat W^n,\hat Y^n,\hat Z^n)$. Estimate \eqref{volterra-dependence} gives
\begin{align*}
 \Vert\phi^n-\phi\Vert_T
 \leq e^{4\lambda T}\Bigl(\Vert c^n-c\Vert_T
   +2\lambda\Vert\phi\Vert_T\int_0^T
     \bigl(|\bar Y^n(s)-\bar Y(s)|
          +|\bar Z^n(s)-\bar Z(s)|\bigr)\,ds\Bigr)\to0\qquad a.s.
\end{align*}
Define $\phi_X^n=-\phi_W^n-\phi_Y^n-\phi_Z^n$ and $\phi_X=-\phi_W-\phi_Y-\phi_Z$. Then
\begin{equation*}
    \Vert\phi_X^n-\phi_X\Vert_T\leq\Vert\phi^n-\phi\Vert_T\longrightarrow0
    \qquad a.s.
\end{equation*}
Thus the four-dimensional solutions converge uniformly on this representation, which proves the convergence in distribution in $D([0,T],\R^4)$. The limiting equations are \eqref{Xhat}--\eqref{Zhat}, and uniqueness follows from Lemma \ref{Volterraintegral}.

If $F$ is continuous, then $c$ and the solution are continuous. If the initial fluctuation vector is Gaussian, its independence from the seven jointly Gaussian driving processes implies that $c$ is Gaussian. The solution is the limit of the linear Picard iterates in the appendix and is therefore Gaussian as well.
\end{proof}

\appendix
\section{The Volterra equation}\label{Appendix}

\begin{lemma}\label{Volterraintegral}
Fix $T>0$ and let $c\in D([0,T],\R^3)$. Let $q,r,\bar X,\bar Y$ be measurable functions on $[0,T]$, bounded in absolute value by $L$. Recall that
\begin{equation*}
 k_W=F^c,\qquad k_Y=H,\qquad k_Z=\beta,
 \qquad k_W+k_Y+k_Z=1.
\end{equation*}
For $\phi=(\phi_W,\phi_Y,\phi_Z)$, put $\phi_X=-\phi_W-\phi_Y-\phi_Z$ and define
\begin{equation}\label{volterra-operator}
\begin{aligned}
 (\mathcal K\phi)(t)
 &=\lambda\int_0^t
 \begin{pmatrix}
 k_W(t-s)\\ k_Y(t-s)\\ k_Z(t-s)
 \end{pmatrix}
 \bigl(q(s)\phi_X(s)+\bar X(s)\phi_Y(s)\bigr)\,ds\\
 &\quad+\lambda
 \begin{pmatrix}0\\-1\\1\end{pmatrix}
 \int_0^t\bigl(r(s)\phi_Y(s)+\bar Y(s)\phi_Z(s)\bigr)\,ds.
\end{aligned}
\end{equation}
Then
\begin{equation}\label{volterra-fixed-point}
 \phi=c+\mathcal K\phi
\end{equation}
has a unique solution in $D([0,T],\R^3)$, and $\mathcal K\phi$ is continuous. For fixed $\bar X,\bar Y$, denote the solution map by
\begin{equation*}
 \Phi:(c,q,r)\longmapsto\phi.
\end{equation*}
This map is Borel measurable when the input path has the $J_1$ Borel sigma-algebra and the coefficients have the $L^1([0,T])$ Borel sigma-algebra.

Let $\phi^n=\Phi(c^n,q^n,r^n)$ and $\phi=\Phi(c,q,r)$, with $|q^n|,|q|,|r^n|,|r|\leq L$ almost everywhere for every $n$. Set
\begin{equation*}
 e_n=\Vert q^n-q\Vert_{L^1} +\Vert r^n-r\Vert_{L^1}.
\end{equation*}
Then
\begin{equation}\label{volterra-dependence}
 \Vert\phi^n-\phi\Vert_T
 \leq e^{4\lambda LT}
       \bigl(\Vert c^n-c\Vert_T+2\lambda\Vert\phi\Vert_T e_n\bigr).
\end{equation}
Consequently, on this bounded coefficient class, $\Phi$ is continuous when the input and output paths have the uniform norm and $q,r$ have the $L^1([0,T])$ norm. In particular,
\begin{equation*}
 \Vert c^n-c\Vert_T+e_n\longrightarrow0
 \quad\Longrightarrow\quad
 \Vert\Phi(c^n,q^n,r^n)-\Phi(c,q,r)\Vert_T\longrightarrow0.
\end{equation*}
\end{lemma}

\begin{proof}
We first bound the operator in \eqref{volterra-operator}. Let $v\in D([0,T],\R^3)$, with $v=(v_W,v_Y,v_Z)$. Note that
\begin{equation*}
    |v_X|=|v_W+v_Y+v_Z|\leq|v|
\end{equation*}
therefore
\begin{align*}
 |q(s)v_X+\bar X(s)v_Y|&\leq2L|v|,\\
 |r(s)v_Y+\bar Y(s)v_Z|&\leq L|v|,
\end{align*}
where $L$ is the bound in the statement of the lemma. The first column vector in \eqref{volterra-operator} has norm one, and $(0,-1,1)$ has norm two. Hence
\begin{equation}\label{volterra-integral-bound}
 |(\mathcal Kv)(t)-(\mathcal Ku)(t)|
 \leq C\int_0^t|v(s)-u(s)|\,ds,
 \qquad C=4\lambda L.
\end{equation}

We also check continuity of the integral terms. Let $k$ be a càdlàg function of bounded variation and let $v$ be bounded and measurable. For $0\leq t\leq t+h\leq T$,
\begin{align*}
 &\left|\int_0^{t+h} k(t+h-s)v(s)\,ds
              -\int_0^t k(t-s)v(s)\,ds\right|\\
 &\quad\leq\Vert v\Vert_T\int_0^t|k(u+h)-k(u)|\,du
                       +h\Vert k\Vert_T\Vert v\Vert_T\\
 &\quad\leq h\Vert v\Vert_T
       \bigl(\Vert k\Vert^{TV}_{(0,T]}+\Vert k\Vert_T\bigr).
\end{align*}
Indeed, writing $|k|$ for the total variation measure,
\begin{align*}
 \int_0^t|k(u+h)-k(u)|\,du
 &\leq\int_0^t\int_{(u,u+h]}|k|(dr)\,du\\
 &\leq h\,|k|((0,T])\\
 &\leq h\,\Vert k\Vert^{TV}_{(0,T]},
\end{align*}
by Tonelli's theorem. The kernels $k_W,k_Y,k_Z$ have bounded variation. Therefore $\mathcal K$ maps bounded measurable functions to Lipschitz continuous functions.

To prove existence in \eqref{volterra-fixed-point}, we use Picard iteration. Set $\phi^{(0)}=c$ and
\begin{equation*}
 \phi^{(m+1)}=c+\mathcal K\phi^{(m)},\qquad m\geq0.
\end{equation*}
By \eqref{volterra-integral-bound},
\begin{equation*}
 |\phi^{(1)}(t)-\phi^{(0)}(t)|
 \leq C\int_0^t|c(s)|\,ds\leq\Vert c\Vert_T Ct.
\end{equation*}
Induction gives
\begin{equation*}
 |\phi^{(m+1)}(t)-\phi^{(m)}(t)|
 \leq\Vert c\Vert_T\frac{(Ct)^{m+1}}{(m+1)!},\qquad m\geq0.
\end{equation*}
Consequently,
\begin{equation*}
 \sum_{m=0}^{\infty}\Vert\phi^{(m+1)}-\phi^{(m)}\Vert_T
 \leq\Vert c\Vert_T(e^{CT}-1)<\infty.
\end{equation*}
Thus the iterates converge uniformly to a càdlàg function
$\phi$, with
\begin{equation}\label{volterra-solution-bound}
 \Vert\phi\Vert_T\leq e^{CT}\Vert c\Vert_T.
\end{equation}
Moreover,
\begin{equation*}
 \Vert\mathcal K\phi^{(m)}-\mathcal K\phi\Vert_T
 \leq CT\Vert\phi^{(m)}-\phi\Vert_T\longrightarrow0.
\end{equation*}
Passing to the limit proves existence.

Now we prove uniqueness. Let $\phi$ and $\psi$ be two solutions to equation \eqref{volterra-fixed-point}. Then
\begin{equation*}
 \sup_{u\leq t}|\phi(u)-\psi(u)|
 \leq C\int_0^t\sup_{v\leq s}|\phi(v)-\psi(v)|\,ds
\end{equation*}
and Gronwall's inequality gives $\phi=\psi$ (uniqueness).

We next prove \eqref{volterra-dependence}. Write $\mathcal K_n$ for the operator with coefficients at index $n$. Subtracting the equations and inserting $\mathcal K_n\phi$ gives
\begin{equation*}
 \phi^n-\phi=(c^n-c)
       +(\mathcal K_n\phi^n-\mathcal K_n\phi)
       +(\mathcal K_n\phi-\mathcal K\phi).
\end{equation*}
The first operator difference satisfies \eqref{volterra-integral-bound} with the same $C$. For the second, the fixed coefficients cancel, so
\begin{align*}
 ((\mathcal K_n-\mathcal K)\phi)(t)
 &=\lambda\int_0^t
 \begin{pmatrix}k_W(t-s)\\k_Y(t-s)\\k_Z(t-s)\end{pmatrix}
       (q^n(s)-q(s))\phi_X(s)\,ds\\
 &\quad+\lambda\begin{pmatrix}0\\-1\\1\end{pmatrix}
       \int_0^t(r^n(s)-r(s))\phi_Y(s)\,ds.
\end{align*}
The norms of the two column vectors are one and two, respectively. Since $|\phi_X(s)|,|\phi_Y(s)|\leq\Vert\phi\Vert_T$,
\begin{align*}
 \Vert\mathcal K_n\phi-\mathcal K\phi\Vert_T
 &\leq\lambda\Vert\phi\Vert_T
       \bigl(\Vert q^n-q\Vert_{L^1}+2\Vert r^n-r\Vert_{L^1}\bigr)\\
 &\leq2\lambda\Vert\phi\Vert_T e_n.
\end{align*}
Therefore,
\begin{equation*}
 \sup_{u\leq t}|\phi^n(u)-\phi(u)|\leq\Vert c^n-c\Vert_T+2\lambda\Vert\phi\Vert_T e_n
                       +C\int_0^t \sup_{r\leq s}|\phi^n(r)-\phi(r)|\,ds
\end{equation*}
and Gronwall's inequality proves \eqref{volterra-dependence}.

Finally, we verify measurability of $\Phi$. Convergence in $J_1$ implies convergence in $L^1([0,T])$: the paths are uniformly bounded and converge at every continuity point of the limit, so dominated convergence applies. For coefficients bounded by $L$,
\begin{equation*}
    \Vert a_nv_n-av\Vert_{L^1}
    \leq L\Vert v_n-v\Vert_{L^1}
    +\Vert v\Vert_T\Vert a_n-a\Vert_{L^1}.
\end{equation*}
Multiplication is therefore continuous from $D([0,T])\times L^1([0,T])$ to $L^1([0,T])$ on this coefficient class, with $J_1$ on $D([0,T])$. Integration against each bounded kernel is continuous from $L^1([0,T])$ to the uniform norm. Since evaluation of a càdlàg path is Borel measurable, induction shows that every evaluation of each Picard iterate is a Borel function of the data. By Theorem 12.5 in \cite{Bill}, the iterates are measurable as $D$-valued maps, and so is their uniform limit $\Phi$. Coefficients equal almost everywhere give the same solution.
\end{proof}

\section*{Acknowledgements}
This study was financed in part by the Coordenação de Aperfeiçoamento de Pessoal de Nível Superior – Brasil (CAPES) – Finance Code $001$. Also, it was partially supported by S\~ao Paulo Research Foundation (FAPESP) under grants $\#2017/10555-0$ and $\#2022/08948-2$.

\section*{Ethical Approval}
Not applicable. This research is purely theoretical and does not involve human participants or animals.

\section*{Declaration of Competing Interest}
The authors declare that they have no known competing financial interests or personal relationships that could have appeared to influence the work reported in this paper.

\section*{Data Availability}
This manuscript has no associated data.


\bibliography{references}

\end{document}